\documentclass[11pt]{amsart}

\makeatletter
\def\paragraph{\@startsection{paragraph}{4}%
  \z@\z@{-\fontdimen2\font}%
  {\normalfont\bfseries}}
\makeatother

\usepackage{amsmath}
\usepackage{amssymb}
\usepackage{amsthm}
\usepackage{enumitem}
\usepackage[final]{microtype}
\usepackage[T1]{fontenc}
\usepackage{array}
\usepackage{dcolumn}
\usepackage[vmargin=1in, hmargin=1.3in]{geometry}
\usepackage{float}
\usepackage{graphicx}
\usepackage{tikz-cd}
\usepackage[maxbibnames=99,style=alphabetic]{biblatex}
\usepackage[hidelinks]{hyperref}
\usepackage{xcolor}
\usepackage[normalem]{ulem}
\usepackage{mathtools}
\usepackage{todonotes}
\usepackage{booktabs}
\usepackage{aliascnt}
\usepackage{subcaption}
\usepackage{xcolor}
\usepackage{hyperref}
\hypersetup{
    colorlinks,
    linkcolor={red!50!black},
    citecolor={blue!50!black},
    urlcolor={blue!80!black}
}

\renewcommand \H{\mathbb{H}}
\renewcommand \v{\vert}
\renewcommand \Re{\mathrm{Re}}
\renewcommand \Im{\mathrm{Im}}
\newcommand \Q{\mathbb{Q}}
\newcommand \R{\mathbb{R}}

\newcommand \Z{\mathbb{Z}}
\newcommand \C{\mathbb{C}}

\newcommand \la{\langle}
\newcommand \ra{\rangle}

\DeclareMathOperator \tr{tr}

\DeclareMathOperator \spn{\text{Span}}
\DeclareMathOperator \Ker{Ker}

\DeclareMathOperator \Aut{Aut}
\DeclareMathOperator \Mod{Mod}

\DeclareMathOperator \GL{GL}

\DeclareMathOperator \SL{SL}
\DeclareMathOperator \PSL{PSL}
\DeclareMathOperator \Sp{Sp}
\DeclareMathOperator \PSp{PSp}

\DeclareMathOperator \PU{PU}

\DeclareMathOperator \Aff{Aff}

\DeclareMathOperator \End{End}

\DeclareMathOperator \Out{Out}
\newcommand{\isom}{\cong}
\DeclareMathOperator{\lcm}{\operatorname{lcm}}

\newcommand \M{\mathcal{M}}
\newcommand \T{\mathcal{T}}
\renewcommand \O{\mathcal{O}}
\newcommand \X{\mathbb{X}}
\newcommand \Xb{\overline{\mathbb{X}}}
\newcommand \Xt{\widetilde{\mathbb{X}}}

\newcommand \F{\mathbb{F}}
\renewcommand \P{\mathbb{P}}
\DeclareMathOperator{\hol}{hol}
\DeclareMathOperator \cc{cc}

\newtheorem{theorem}{Theorem}
\numberwithin{theorem}{section}
\newaliascnt{lemma}{theorem}
\newtheorem{lemma}[lemma]{Lemma}
\aliascntresetthe{lemma}

\newaliascnt{proposition}{theorem}
\newtheorem{proposition}[proposition]{Proposition}
\aliascntresetthe{proposition}

\newaliascnt{corollary}{theorem}

\aliascntresetthe{corollary}

\newaliascnt{conjecture}{theorem}

\aliascntresetthe{conjecture}

\newaliascnt{definition}{theorem}
\newtheorem{definition}[definition]{Definition}
\aliascntresetthe{definition}

\newaliascnt{observation}{theorem}
\newtheorem{observation}[observation]{Observation}
\aliascntresetthe{observation}

\newaliascnt{example}{theorem}

\aliascntresetthe{example}

\newaliascnt{problem}{theorem}
\newtheorem{problem}[problem]{Problem}
\aliascntresetthe{problem}

\newaliascnt{remark}{theorem}
\theoremstyle{remark}
\newtheorem{remark}[remark]{Remark}
\aliascntresetthe{remark}

\newaliascnt{question}{theorem}
\theoremstyle{remark}
\newtheorem{question}[question]{Question}
\aliascntresetthe{question}

\numberwithin{figure}{section}

\numberwithin{equation}{section}

\begin{document}
\title{Veech fibrations}
\author{Sam Freedman}
\author{Trent Lucas}
\address{University of Chicago}
\email{sfreedman67@uchicago.edu}
\address{University of California, Irvine}
\email{trent.lucas@uci.edu}
\date{\today}
\begin{abstract}
    We investigate complex surfaces that fiber over Teichm\"uller curves where the generic fiber is a Veech surface.
    When the fiber has genus one, these surfaces are elliptic fibrations; for higher genus fibers, they are typically minimal surfaces of general type.
    We compute the topological and complex-geometric invariants of these surfaces using the monodromy action on the mod $m$ homology of the fiber.
    We get exact values of the invariants for all known algebraically primitive Teichm\"uller curves.
\end{abstract}
\maketitle

\section{Introduction}

\subsection{Main problem and results}\label{subsec:main-results}
In this paper, we study certain compact complex surfaces that arise from the theory of translation surfaces; we call these complex surfaces \emph{Veech fibrations}.
Veech fibrations are special examples of \emph{Lefschetz fibrations}, which means that each Veech fibration is equipped with a map onto a closed (real) surface that has finitely many nodal singular fibers.
The smooth fibers are certain types of highly symmetric translation surfaces known as \emph{Veech surfaces}.

Recall that a \emph{translation surface} is obtained by identifying pairs of parallel edges of a Euclidean polygon in $\R^2$ using translations.
This is equivalent to the data of a pair $(X,\omega)$, where $X$ is a Riemann surface and $\omega$ is a holomorphic $1$-form on $X$.
(See, for example, the book~\cite{AthreyaMasur} for background on translation surfaces.)
There is a natural action of $\SL(2,\R)$ on the set of all translation surfaces induced by its linear action on $\R^2$.
A \emph{Veech surface} is a translation surface whose $\SL(2,\R)$-orbit determines a \emph{Teichm\"uller curve} $C \to \M_g$, a finite volume hyperbolic (real) surface that is isometrically immersed in the moduli space of genus $g$ Riemann surfaces $\M_g$.

Veech surfaces are rare and remarkable objects; starting with the work of Veech \cite{Veech}, their study has led to a rich interplay between dynamics and the geometry of $\M_g$.
Given the exceptional nature of Veech surfaces, we expect Veech fibrations to be rich examples $4$-manifold geometry and topology.
The purpose of this paper is to investigate the following.

\begin{problem}\label{prob:geography-problem}
  Compute the topological and geometric invariants of a Veech fibration in terms of the geometry of its smooth fibers.
\end{problem}

We are especially interested in \emph{explicit} solutions to \autoref{prob:geography-problem}, and in particular whether any familiar complex surfaces arise as Veech fibrations.

Following M\"oller~\cite{moller_variations_2005}, a Veech surface $(X,\omega)$ determines a collection of Veech fibrations in the following manner:
\begin{enumerate}
    \item Choose a finite manifold cover $\pi:\widetilde{\M}_g \rightarrow \M_g$.
    These correspond to finite index torsion-free subgroups of the mapping class group $\pi_0(\mathrm{Homeo}^+(X))$.
    \item The Teichm\"uller curve $C \rightarrow \M_g$ associated to $(X,\omega)$ lifts to an isometrically immersed finite volume hyperbolic surface $B_\pi \rightarrow \widetilde{\M}_g$.
    \item The cover $\widetilde{\M}_g$ admits a universal family which pulls back to a smooth surface bundle $\X_\pi \rightarrow B_\pi$.
    Concretely, $\X_\pi$ is the ``tautological family'': each point $b \in B_\pi$ corresponds to a Veech surface $(X',\omega')$ in the $\SL(2,\R)$-orbit of $(X,\omega)$, and the fiber of $\X_\pi$ over $b$ is precisely $X'$.
    \item Over a compactification $\overline{B}_\pi$ of $B_\pi$, the family $\X_\pi \rightarrow B_\pi$ extends to a family of \emph{semistable} Riemann surfaces $\Xt_\pi \rightarrow \overline{B}_\pi$ whose total space is a smooth compact complex surface.
    We call $\Xt_\pi$ the \emph{Veech fibration} associated to the cover $\pi$.
\end{enumerate}

Using results of Chen and M\"oller \cite{ChenMoller,Moller_PCMI}, we will derive formulas for the invariants of $\Xt_\pi$ in terms of $(X,\omega)$ and the chosen cover $\pi:\widetilde{\M}_g \rightarrow \M_g$.
However, computing the invariants \emph{explicitly} is difficult since it depends on an explicit description of the cover of the Teichm\"uller curve $B_\pi \rightarrow C$.

To make the computation tractable, we restrict our focus to the natural family of \emph{congruence covers} $\widetilde{\M}_g[m] \rightarrow \M_g$ for $m \geq 3$ corresponding to the \emph{principal congruence subgroups} of the mapping class group.
The cover $\widetilde{\M}_g[m]$ has an associated \emph{congruence Veech fibration} $\Xt_m \rightarrow \overline{B}_m$.
Computing the cover of the Teichm\"uller curve $B_m \rightarrow C$ amounts to computing the image of the monodromy representation
\begin{equation*}
  \rho_m:\Aff^+(X,\omega) \rightarrow \Aut(H_1(X; \Z/m\Z)),
\end{equation*}
where $\Aff^+(X,\omega)$ is the group of affine automorphisms of $X$ (that is, the group of orientation-preserving homeomorphisms preserving the flat metric given by $\omega$).
This problem is still difficult, as it is related to the poorly understood problem of computing \emph{Kontsevich--Zorich monodromy}, that is, the action of $\Aff^+(X,\omega)$ on $H_1(X;\Z)$.

In this paper, we compute the image of $\rho_m$ for infinitely many $m$ for several families of Veech surfaces, which yields many explicit solutions to \autoref{prob:geography-problem}.
For example, we solve \autoref{prob:geography-problem} for certain ``regular polygon surfaces'' first studied by Veech \cite{Veech}: 

\begin{theorem}\label{thm:intro-double-odd-gons}
  Let $(X,\omega)$ be the regular $q$-gon surface of genus $g = (q - 1)/2$ where $q \ge 5$ is prime.
  Fix a prime $p \geq 3$ for which the minimal polynomial of $4\cos\left(\pi/q\right)^2$ is irreducible over $\F_p$, and let $\Xt_{q,p} \rightarrow \overline{B}_{q,p}$ be the $p$-congruence Veech fibration.
  Then: 
  \begin{itemize}
    \item $\Xt_{q,p}$ is a minimal general type surface,
    \item the map $\pi_1(\Xt_{q, p}) \to \pi_1(\overline{B}_{q, p})$ is an isomorphism,
    \item $\Xt_{q,p}$ has Euler characteristic and signature
    \[
      e(\Xt_{q, p}) = dg + d(q-3)\left(\frac12 - \frac1q - \frac1p\right) \quad \text{and} \quad  \sigma(\Xt_{q, p}) = -\frac{d(q^2 - 1)}{4q},
    \]
    where $d = |\PSL(2, 5)|$ if $(p, g) = (3, 2)$ or $|\PSL(2, \F_{p^g})|$ otherwise.
  \end{itemize}
\end{theorem}

See \autoref{subsec:about-the-proofs} for motivation for our hypotheses on the prime $p$.
A particularly interesting consequence of \autoref{thm:intro-double-odd-gons} is that the Veech fibration $\Xt_{5, 3}$ is a well-known Horikawa surface; see \autoref{rem:double_pentagon}.

Our methods for understanding the representation $\rho_m$ for the case of the double polygon surface extend to all \emph{algebraically primitive} Veech surfaces.
These are Veech surfaces for which the action of an affine automorphism $\phi \in \Aff^+(X,\omega)$ on $H_1(X;\R)$ is ``tautological'' in a precise sense.
There are three known families of algebraically primitive Veech surfaces:
\begin{enumerate}[label=(\roman*)]
  \item the genus $2$ \emph{Weierstrass eigenforms} in the \emph{minimal stratum} with nonsquare discriminant, which were classified by McMullen~\cite{McMullen_Billiards_Tcurves_Hilbert_modular_surfaces}.
  \item the regular $n$-gon surfaces studied by Veech \cite{Veech} where $n$ is prime, twice a prime, or a power of $2$,
  \item the ``sporadic'' Veech surfaces $E_6$ and $E_7$ discovered by Leininger~\cite{Leininger_2004}.
\end{enumerate}
We solve \autoref{prob:geography-problem} for all these families.

\begin{theorem}\label{thm:all-alg-prim-results}
  Let $(X, \omega)$ be a Veech surface in one of the families (i)-(iii) described above, and for $m \geq 3$, let $\Xt_m \rightarrow \overline{B}_m$ be the $m$-congruence Veech fibration.
  Then for a prime $p \ge 3$ in a computable set of primes depending on $(X, \omega)$:
  \begin{itemize}
    \item $\Xt_p$ is a minimal surface of general type,
    \item the map $\pi_1(\Xt_p) \rightarrow \pi_1(\overline{B}_p)$ is an isomorphism,
    \item formulas for the Euler characteristic and signature of $\Xt_p$ are given in Table \ref{tab:l-tables} for family (i), Tables \ref{tab:BM_q_and_2q}-\ref{tab:BM_2tok} for family (ii), and Table \ref{tab:sporadics} for family (iii).
  \end{itemize}
\end{theorem}

\subsection{About the proofs}\label{subsec:about-the-proofs}
For any Veech surface $(X,\omega)$, there is a \emph{holonomy homomorphism} $\hol:H_1(X,\Z) \rightarrow \R^2$ given by integrating the $1$-form $\omega$.
In the algebraically primitive case, this map is an embedding, which makes it tractable to compute the action $\rho:\Aff^+(X,\omega) \rightarrow \Aut(H_1(X;\Z))$.
In principle, one could use ``strong approximation'' (see, for example, Matthews--Vaserstein--Weisfeiler \cite{Matthews_Vaserstein_Weisfeiler}) to compute
the image of the map $\rho_p:\Aff^+(X,\omega) \rightarrow \Aut(H_1(X;\F_p))$ for all but finitely many primes, but it is not clear how to explicitly determine the exceptional primes.
Our main technical result is an \emph{effective} criterion on a prime $p$ that allows us to determine the image of $\rho_p$.

The key is that for an integral linear combination $\alpha = \sum c_i \tr(D\phi_i)$ where $c_i \in \Z$, $\phi_i \in \Aff^+(X,\omega)$, and $D\phi_i \in \SL(2,\R)$ is the derivative of $\phi_i$, we can construct a homomorphism $\Z[\alpha] \rightarrow \End(H_1(X;\Z))$ that corresponds to the usual multiplication by $\alpha$ under the holonomy homomorphism\footnote{The existence of the map $\Z[\alpha] \rightarrow \End(H_1(X;\Z))$ is a shadow of a deeper theorem of McMullen~\cite{McMullen_Billiards_Tcurves_Hilbert_modular_surfaces} and M\"oller~\cite{moller_variations_2005}, which says that the Jacobian of an algebraically primitive Veech surface admits real multiplication by its trace field.}.
This gives $H_1(X;\Z)$ the structure of a $\Z[\alpha]$-module.
Using this structure, we prove the following main technical result.

\begin{theorem}\label{thm:main_technical}
  Let $(X,\omega)$ be an algebraically primitive Veech surface of genus $g$.
  Suppose that there exist affine automorphisms $\phi_H,\phi_V \in \Aff^+(X,\omega)$ whose derivatives are given by
  \begin{equation*}
  D\phi_H = \begin{pmatrix}
  1 & c \\ 0 & 1
  \end{pmatrix},\ 
  D\phi_V = \begin{pmatrix}
  1 & 0 \\ d& 1
  \end{pmatrix}.
  \end{equation*}
  Let $\alpha \coloneqq cd = \tr(D\phi_HD\phi_V) - \tr(I)$.
  
  Suppose that the image of $\hol:H_1(X,\Z) \rightarrow \R^2$ is a free $\Z[\alpha]$-module of rank $2$ with a $\Z[\alpha]$-basis of the form $\{\left( \begin{smallmatrix} x \\ 0 \end{smallmatrix} \right), \left( \begin{smallmatrix} 0 \\ y \end{smallmatrix} \right)\}$.
  Let $p \geq 3$ be a prime such that the minimal polynomial $m_\alpha(x)$ of $\alpha$ is irreducible over $\F_p$ and $(p,g) \neq (3,2)$.  Then,
  \begin{equation*}
  \Im\left(\Aff^+(X,\omega) \xrightarrow{\rho_p} \Aut(H_1(X;\F_p))\right) \cong \SL(2,\F_{p^g}).
  \end{equation*}
\end{theorem}

To apply \autoref{thm:main_technical} to the families of \autoref{thm:all-alg-prim-results}, we use the fact that all these surfaces arise from the \textit{Thurston--Veech construction} (see \autoref{subsec:TV_construction}), providing a concrete way of understanding the image of the holonomy homomorphism.

\subsection{Connection to elliptic fibrations}\label{subsec:intro--elliptic-fibrations}
If $X$ has genus 1, the associated Veech fibrations are examples of \emph{elliptic fibrations}.
In fact, the Veech fibrations $\Xt \rightarrow \overline{B}$ obtained from the construction in \autoref{subsec:main-results} are known as \emph{elliptic modular surfaces}, first studied by Kodaira~\cite{Kodaira} and later by Shioda~\cite{Shioda}.

More concretely, any genus $1$ Veech surface $(X,\omega)$ is $\SL(2,\R)$-equivalent to the square torus, and the associated Teichm\"uller curve $C$ is in fact the full moduli space $\M_1$.
A choice of cover $\widetilde{\M}_g \rightarrow \M_g$ is equivalent to a choice of a finite index subgroup $\Gamma \subseteq \SL(2,\Z)$, and the congruence cover $\widetilde{\M}_g[m]$ corresponds to the principal congruence subgroup $\Gamma(m) \subseteq \SL(2,\Z)$.
In this setting, the natural map $\rho:\Aff^+(X,\omega) \rightarrow \Aut(H_1(X;\Z))$ is an isomorphism onto $\SL(2,\Z)$, and hence the image of the map $\rho_m:\Aff^+(X,\omega) \rightarrow \Aut(H_1(X;\Z/m\Z))$ is isomorphic to $\SL(2, \Z/m\Z)$.

For any choice of cover $\widetilde{\M}_1 \rightarrow \M_1$ corresponding to a finite index torsion-free subgroup $\Gamma \subseteq \SL(2,\Z)$, Shioda \cite{Shioda} computes the singular fibers and geometric genus of the associated Veech fibration $\Xt \rightarrow \overline{\M}_1$ in terms of the group $\Gamma$.
In the case where $\widetilde{\M}_1$ is the congruence cover $\widetilde{\M}_1[m]$,
Shioda also explicitly computes the group of sections of the congruence Veech fibration $\Xt_m \rightarrow \overline{B}_m$.
See \cite{Sebbar_Besrour_2021} for a survey of some elliptic surfaces that arise from low-index subgroups of $\SL(2,\Z)$.

\subsection{Simply connected Veech fibrations}
A natural problem is to determine which simply connected 4-manifolds arise as Veech fibrations.
A necessary condition for $\Xt$ to be simply connected is that the base $\overline{B}$ has genus 0 (see \autoref{prop:pi_1_total_space}).
For the Veech fibration $\Xt_m \rightarrow \overline{B}_m$ associated to the congruence cover $\widetilde{\M}_g[m] \rightarrow \M_g$, such examples are difficult to come by.
The cover of the Teichm\"uller curve $B_m \rightarrow C$ must be branched over every orbifold point of $C$, and the degree tends to be large. However, the only high-degree regular branched covers $S^2 \rightarrow S^2$ are cyclic or dihedral, and such covers have constrained branch points.

In the case where $X$ has genus 1, the cover $B_m$ has genus 0 for $m \in \{3,4,5\}$.
In these cases, the results of \autoref{sec:top_invariants} tell us that $\Xt_3$, $\Xt_4$, and $\Xt_5$ are the smooth 4-manifolds $E(1)$, $E(2)$, and $E(5)$ (see, for example, \cite[Ch.~3.1]{gompf_stipsicz} for the definition of these 4-manifolds).
In particular, $\Xt_3$ is a Beauville rational surface and $\Xt_4$ is an elliptic K3 surface \cite{Sebbar_Besrour_2021}.

As mentioned above, an interesting simply connected example is the Veech fibration $\Xt_{5,3} \rightarrow \overline{B}_{5,3}$ of \autoref{thm:intro-double-odd-gons} associated to the double pentagon Veech surface.
In this case, $\Xt_{5,3}$ is a Horikawa surface with Euler characteristic 116 and signature $-72$.
Up to deformation equivalence, there are two such surfaces; they are known to be homeomorphic, and it is a famous open problem to determine if they are diffeomorphic (see, for example, \cite{Auroux}).
See \autoref{rem:double_pentagon} for details.

Computer experiments show that there are more simply connected examples.   
Determining which Teichm\"uller curves yield simply connected Veech fibrations appears to be a subtle problem.
We know that the deck group of the cover $B_m \to C$ is cyclic, dihedral, $A_4$, $S_4$, or $A_5$, and that the cover is branched over the orbifold points of $C$. 
However, the cover need not branch over every cusp of $C$, making it hard to describe the resulting constraints on $C$.

\subsection{Previous work}
Our construction of Veech fibrations was originally due to M\"oller \cite{moller_variations_2005}.
Since then, Veech fibrations have been used to prove a variety of deep results about Teichm\"uller curves, see, for example, \cite{Bouw-Moller, ChenMoller, moller_periodic_2006, bonnafoux2022arithmeticity}.
Although Veech fibrations have been useful tools in the study of Teichm\"uller curves, our work is a first step toward studying them as interesting spaces in their own right.

In their thesis, Ronkin \cite{ronkin} also studied families over Teichm\"uller curves using this construction.
They evaluated the Miller-Morita-Mumford class $\kappa_1$ on Teichm\"uller curves, obtaining formulas for the area of the Teichm\"uller curve with respect to the Weil--Peterson and Teichm\"uller metrics.

As mentioned above, Veech fibrations are examples of Lefschetz fibrations, a well-studied class of 4-manifolds (see, for example, Gompf--Stipsicz \cite[Ch~8]{gompf_stipsicz}).
Smith \cite{Smith} gives a formula for the signature of Lefschetz fibrations that agrees with \autoref{prop:signature_formula}; see \autoref{rmk:smith-formula} for more details.

For another approach to studying representations $\rho_m$, see Finster \cite{Finster2016}.

In recent years there has been an increase in interest in 4-manifolds and their mapping class groups from a hands-on geometric viewpoint; see, e.g., Farb--Looijenga \cite{Farb_Looijenga_2022} and Lee \cite{Seraphina_Lee_2023}.
Veech fibrations might serve as interesting examples for such investigations in future work.
\subsection{Contents}
In \autoref{sec:background}, we recall some background on Veech surfaces, Teichm\"uller curves and the action of $\Aff^+(X,\omega)$ on $H_1(X)$.
In \autoref{sec:veech-fibrations}, we define Veech fibrations following M\"oller, and describe their structure as families of curves and more generally as Lefschetz fibrations.
We also describe the \emph{congruence cover} of Teichm\"uller curves and the resulting \emph{congruence Veech fibrations}.
In \autoref{sec:top_invariants}, we describe the topological invariants of Veech fibrations: their Euler characteristic, signature, fundamental group, Betti numbers, and intersection form.
In \autoref{sec:geometric_invariants}, we discuss Veech fibrations as complex surfaces.  We show that most Veech fibrations are minimal general type surfaces, and we provide a sufficient condition to show that their BMY inequality is strict.

In \autoref{sec:alg_prim_Tcurves}, we address the case that $(X,\omega)$ is algebraically primitive.
We apply the results of \autoref{sec:top_invariants} and \autoref{sec:geometric_invariants} to deduce some general results on Veech fibrations in the algebraically primitive case.
Then, we prove \autoref{thm:main_technical}, which gives us a criterion to compute the degree of the congruence cover.  In \autoref{subsec:l_tables}, \autoref{subsec:b-m}, and \autoref{subsec:sporadic}, we apply \autoref{thm:main_technical} to the genus 2 Weierstrass eigenforms, regular polygon surfaces, and sporadic examples.

\subsection{Acknowledgments}
We thank Benson Farb for his mini-course on mapping class groups of 4-manifolds that initially inspired this project, and for offering many key ideas at its outset.
We thank Dawei Chen for his comments and for suggesting the minimality argument in \autoref{rem:double_pentagon}.
We thank our advisors Jeremy Kahn and Bena Tshishiku for reading earlier drafts and for many helpful discussions.
Finally, we thank Dan Abramovich, Dori Bejleri, Aaron Calderon, Pat Hooper, Carlos Matheus, Curt McMullen, Martin M\"oller, and Karl Winsor for their comments and for answering our questions.

\section{Background}\label{sec:background}

In \autoref{subsec:Veech_surfaces} and \autoref{subsec:teichmuller_curves} we recall background on Veech surfaces and Teichm\"uller curves, mainly following McMullen's survey \cite{McMullen_survey_2022}.
We then describe how the affine automorphism group acts on the homology of a Veech surface in \autoref{subsec:action-on-homology}.

\subsection{Veech surfaces}\label{subsec:Veech_surfaces}
\subsubsection{Veech surfaces}
A \emph{translation surface} is a pair $(X,\omega)$ where $X$ is a Riemann surface and $\omega$ is a holomorphic 1-form on $X$.
We let $\Aff^+(X,\omega)$ denote the group of orientation preserving affine automorphisms, i.e.\ orientation preserving homeomorphisms of $X$ which are given by affine maps in the natural coordinates of $\omega$.
There is a well-defined derivative map $D:\Aff^+(X,\omega) \rightarrow \SL(2,\R)$.
The image of this map, denoted $\SL(X,\omega)$, is called the \emph{Veech group} of $X$.
We say $(X,\omega)$ is a \emph{Veech surface} if $\SL(X,\omega)$ is a lattice in $\SL(2,\R)$.
For any $\phi \in \Aff^+(X,\omega)$ with derivative $A = D\phi \in \SL(X,\omega)$, we have that
\begin{equation}\label{eqn:tautological_action}
    \phi^*(\omega) = 
    \begin{pmatrix}
        1 & i
    \end{pmatrix}
    A
    \begin{pmatrix}
        \Re(\omega) \\ \Im(\omega)
    \end{pmatrix}.
\end{equation}
Hence we get a short exact sequence
\begin{equation*}
    1 \rightarrow \Aut(X,\omega) \rightarrow \Aff^+(X,\omega) \rightarrow \SL(X,\omega) \rightarrow 1
\end{equation*}
where $\Aut(X,\omega)$ is the (finite) group of complex automorphisms of $X$ that preserve $\omega$.

Throughout this work, we let $K$ denote the \emph{trace field} of $\SL(X,\omega)$, i.e.\ the subfield of $\R$ generated by the traces of elements of $\SL(X,\omega)$.
For a Veech surface, the trace field is a totally real number field of degree at most $g$.  Moreover, the trace of any $A \in \SL(X,\omega)$ is an algebraic integer.

\subsubsection{Cylinder decompositions and parabolic elements}
A \emph{cylinder decomposition} of $(X,\omega)$ is a decomposition of $X$ into parallel cylinders whose edges are comprised of saddle connections (that is, Euclidean straight line geodescis joining two zeros of $\omega$ with no zeros in their interior).
The \emph{modulus} of a cylinder is the ratio of its height over circumference.
On a Veech surface $(X,\omega)$, the moduli of the cylinders in any cylinder decomposition will have rational ratios.

It follows that if $(X,\omega)$ is a Veech surface, then for any cylinder decomposition of $(X,\omega)$, there is an associated parabolic subgroup of $\SL(X,\omega)$ that contains a multitwist around the core curves of the cylinders.
More concretely, if the cylinders have moduli $m_1, \ldots, m_k$ and $m$ is a real number such that $m_i/m \in \Z$ for all $i$, then the associated parabolic subgroup of $\SL(X,\omega)$ contains $D\phi$, where $\phi \in \Aff^+(X,\omega)$ is an $m_i/m$-fold Dehn twist around the $i$th cylinder's core curve.
If we choose $m$ such that the integers $m_i/m$ are coprime, we call $\phi$ the \emph{minimal multitwist} about the cylinder decomposition.
If the cylinders are horizontal, then $D\phi$ is the matrix
\begin{equation*}
    \begin{pmatrix}
        1 & \frac{1}{m} \\ 0 & 1
    \end{pmatrix}.
\end{equation*}
Conversely, if $A \in \SL(X,\omega)$ is a parabolic element fixing a line in $\R^2$ of slope $s$, then $(X,\omega)$ has a cylinder decomposition where each cylinder has slope $s$.

We say that a Veech surface $(X, \omega)$ is a \textit{Thurston--Veech surface} if it is horizontally and vertically periodic.
Up to an element of $\SL(2, \R)$, any Veech surface is a Thurston--Veech surface.

\subsection{Teichm\"uller curves}\label{subsec:teichmuller_curves}
Fix a genus $g$ Veech surface $(X,\omega)$.
The group $\SL(2,\R)$ naturally acts on the Teichm\"uller space of holomorphic 1-forms $\Omega\T_g$, and the orbit of $(X,\omega)$ descends to a holomorphic embedding $\H \hookrightarrow \T_g$ of the upper half plane into Teichm\"uller space.
The group $\Aff^+(X,\omega)$ naturally embeds in the mapping class group $\Mod(X) \coloneqq \pi_0(\textrm{Diff}^+(X))$, and the image of this embedding is the $\Mod(X)$-stabilizer of the set $\H$.
The subgroup $\Aut(X,\omega)$ maps into the pointwise stabilizer of $\H$.
It follows that the geodesic $\H \hookrightarrow \T_g$ descends to a map $C := \H/\SL(X,\omega) \rightarrow \M_g$.
Since $\SL(X,\omega)$ is a lattice, the quotient $C$ is a finite volume hyperbolic surface, called the \emph{Teichm\"uller curve} generated by $(X,\omega)$.
The map $C \rightarrow \M_g$ is proper and generically injective, and the normalization of the image is $C$.

The Veech group is never cocompact; the cusps of $C$ correspond to conjugacy classes of maximal parabolic subgroups of $\SL(X,\omega)$, and hence also to $\Aff^+(X,\omega)$-orbits of cylinder decompositions.
Up to conjugation, the subgroup associated to a cusp can have one of three forms:
\begin{equation*}
    P = \left\la
    \begin{pmatrix}
        1 & a \\ 0 & 1
    \end{pmatrix}
    \right\ra,\ 
    P = \left\la
    \pm\begin{pmatrix}
        1 & a \\ 0 & 1
    \end{pmatrix}
    \right\ra,\ 
    \text{or }
    P = \left\la
    -\begin{pmatrix}
        1 & a \\ 0 & 1
    \end{pmatrix}
    \right\ra. 
\end{equation*}
In the first two cases, we say that the cusp is \emph{regular}, and in the third case, we say that it's \emph{irregular}.
We note that the cusp generator is not necessarily the derivative of the multitwist constructed above.
For instance, the Teichm\"uller curve of the Eierlegende Wollmilschau (see, e.g., \cite{forni_matheus_survey}) has a single cusp, and the cusp generator is the derivative of a $\frac{1}{4}$-twist around two horizontal cylinders.

Note that the Veech group $\SL(X,\omega)$ may contain torsion (such as the matrix $-I$), so $C$ is in general an orbifold.  
It is at times convenient to work with the alternative orbifold structure $\widehat{C} \coloneqq \H/\PSL(X,\omega)$, where $\PSL(X,\omega)$ is the image of $\SL(X,\omega)$ in $\PSL(2,\R)$.
That is, $\widehat{C}$ is the hyperbolic orbifold associated to $\SL(X,\omega)$.

\subsection{Action on homology of Veech surfaces}\label{subsec:action-on-homology}
Consider the natural action on homology
\begin{equation*}
    \rho:\Aff^+(X,\omega) \rightarrow \Sp(H_1(X;\Z)).
\end{equation*}
The representation $\rho$ is injective and generally not surjective, both of which can be seen from a natural splitting of $H_1(X;\R)$ that is preserved by $\Aff^+(X,\omega)$.
This splitting is obtained as follows.
First, we have the \emph{tautological subspace} $H_1(X)^{st} \subseteq H_1(X;\R)$, defined as the dual of the 2-dimensional subspace
\begin{equation*}
    \spn_\R\{\Re(\omega), \Im(\omega)\} \subseteq H^1(X;\R).
\end{equation*}
This subspace is in fact defined over the trace field $K$ of $\SL(X,\omega)$ (see, e.g., \cite{Moller_PCMI}).
If we let $\sigma_1, \ldots, \sigma_d$ denote the $d$ embeddings $K \hookrightarrow \R$, we get a splitting
\begin{equation}\label{eqn:splitting}
    H_1(X;\R) = \bigoplus_{i=1}^d \left(H_1^{st}(X)^{\sigma_i}\right) \oplus H_1^{(0)}(X),
\end{equation}
where $H_1^{st}(X)^{\sigma_i}$ denotes the Galois conjugate of $H_1^{st}(X)$, and $H_1^{(0)}(X)$ is the symplectic complement of the subspaces $H_1^{st}(X)^{\sigma_i}$.
In fact, M\"oller \cite{moller_variations_2005} proved that this gives a splitting of the variation of Hodge structures over $C$.
The subspace $H_1^{(0)}(X)$ can be taken to be the kernel of the \emph{holonomy homomorphism}
\begin{align*}
    \hol:H_1(X;\Z) &\rightarrow \R^2 \cong \C, \quad \gamma \mapsto \int_\gamma \omega.
\end{align*}
We study the holonomy map in more detail in \autoref{sec:alg_prim_Tcurves}.
Note that the splitting (\ref{eqn:splitting}) is generally not defined over $\Z$, and therefore does not necessarily descend to a splitting of $H_1(X;\Z/m\Z)$.

The action of $\Aff^+(X,\omega)$ on $H_1^{st}(X)$ is ``tautological'' in the sense that, with respect to the basis $\{\Re(\omega), \Im(\omega)\}$, any $\phi \in \Aff^+(X,\omega)$ acts by the matrix $D\phi$.
On the other hand, the action of $\Aff^+(X,\omega)$ on $H_1^{(0)}(X)$, also called the \emph{Kontsevich--Zorich monodromy}, is generally difficult to understand.
Many authors have studied the Kontsevich--Zorich monodromy; see e.g.\ \cite{forni_matheus_survey} for a survey or \cite{bonnafoux2022arithmeticity} for a recent article.

There are two extreme cases of the splitting (\ref{eqn:splitting}):
\begin{enumerate}[label=(\roman*)]
    \item the case that $(X,\omega)$ is \emph{algebraically primitive}, meaning that $K$ has degree $g$ and
    \begin{equation*}
        H_1(X;\R) = \bigoplus_{i=1}^g H_1^{st}(X)^{\sigma_i},
    \end{equation*}
    \item the case that $(X,\omega)$ is square-tiled, in which case $K = \Q$ and
    \begin{equation*}
        H_1(X;\R) = H_1^{st}(X) \oplus H_1^{(0)}(X).
    \end{equation*}
\end{enumerate}
In case (ii), the splitting is in fact defined over $\Q$, and hence descends to $H_1(X;\Z/m\Z)$ for all but finitely many $m$.

\section{Veech fibrations}\label{sec:veech-fibrations}

We now describe how to construct families of Veech surfaces over Teichm\"uller curves.
In \autoref{subsec:Defining_Veech_fibrations} we address a general construction originally due to M\"oller, and describe the resulting family as a Lefschetz fibration.
In \autoref{subsec:congruence_Veech_fibrations} we specialize to those fibrations corresponding to the \textit{congruence subgroups} of $\Aff^+(X, \omega)$, i.e., the affine automorphisms that act trivially on $H_1(X; \Z/m\Z)$.

\subsection{Defining Veech fibrations}\label{subsec:Defining_Veech_fibrations}

\subsubsection{The general construction}\label{subsec:the_general_construction}
Fix a genus $g$ Veech surface $(X,\omega)$, and let $C \rightarrow \M_g$ be the associated Teichm\"uller curve.  
One might hope to build a family of curves $\X \rightarrow C$ by restricting the ``universal family'' $\M_{g,1} \rightarrow \M_g$, but this construction in general in a stack since $X$ may have automorphisms.
To obtain an honest manifold, we must pass to a finite cover of $\M_g$ corresponding to a torsion free finite index subgroup of $\Mod(X)$.

More concretely, let $\Gamma' \subseteq \Mod(X)$ be a torsion free finite index subgroup, and let $\Gamma = \Aff^+(X,\omega) \cap \Gamma'$.  Then $\widetilde{\M}_g \coloneqq \T_g/\Gamma'$ is a finite cover of $\M_g$, and $B \coloneqq \H/\Gamma$ is a finite cover of $C$ (as a cover of orbifolds, it is branched only at the cusps of $C$).
Moreover, the map $C \rightarrow \M_g$ lifts to a map $B \rightarrow \widetilde{\M}_g$.
Since $\Gamma'$ is torsion free, $\widetilde{\M}_g$ has a universal family $\X_{\mathrm{univ}} \rightarrow \widetilde{\M}_g$, which pulls back to a family $\X \rightarrow B$.

We now want to complete the family $\X \rightarrow B$ to a family over the closure $\overline{B}$.
If we assume that the monodromy around each cusp of $B$ is a unipotent element of $\GL(H_1(X;\Z))$, then the family $\X \rightarrow B$ extends to a family of stable curves $\Xb \rightarrow \overline{B}$.
(See the work of Cautis \cite[Thm~1.1]{Cautis} who summarizes the work of Deligne, Mumford, and Grothendieck in this direction).
The resulting total space $\Xb$ need not be smooth: the fibers over the cusps of $B$ may have nodal singularities of type $A_{n-1}$, i.e.\ singularities locally modeled by $xy = t^n$.  
To obtain a smooth total space, we can use blow ups to replace each singularity with a chain of $n-1$ $\P^1$'s with self-intersection $-2$, resulting in $n$ nodes modeled on $xy = t$; see the book of Harris and Morrison \cite{HarrisMorrison} for more information.
The resulting family, which we denote $\Xt \rightarrow \overline{B}$, now has \textit{semistable fibers} over the cusps of $B$.
The whole construction is depicted in the following diagram:
\begin{equation*}
    \begin{tikzcd}
	\Xt & \Xb & \X & {\X_{\mathrm{univ}}} \\
	& {\overline{B}} & B & {\widetilde{\M}_g} \\
	&& C & {\M_g}
	\arrow[from=3-3, to=3-4]
	\arrow[from=2-3, to=2-4]
	\arrow[from=1-4, to=2-4]
	\arrow[from=1-3, to=2-3]
	\arrow[from=2-3, to=3-3]
	\arrow[from=2-4, to=3-4]
	\arrow[from=1-2, to=2-2]
	\arrow[from=1-1, to=2-2]
	\arrow[from=1-1, to=1-2]
	\arrow[from=1-3, to=1-4]
	\arrow[hook', from=2-3, to=2-2]
	\arrow[hook', from=1-3, to=1-2]
\end{tikzcd}
\end{equation*}

We summarize the construction with a definition:
\begin{definition}\label{def:Veech_fibration}
    Let $(X,\omega)$ be a Veech surface with Teichm\"uller curve $C \rightarrow \M_g$.
    Let $B \to C$ be a finite cover such that
    \begin{enumerate}[label=(\roman*)]
        \item $\pi_1(B) = \Aff^+(X,\omega) \cap \Gamma'$ for some torsion free finite index subgroup $\Gamma' \subseteq \Mod(X')$,
        \item the monodromy about each cusp of $B$ is a unipotent element of $\GL(H_1(X;\Z))$.
    \end{enumerate}
    We call the resulting semistable family of curves $\Xt \rightarrow \overline{B}$ a \emph{Veech fibration} with fiber $(X,\omega)$.
\end{definition}

If one instead has a cover $B \rightarrow C$ satisfying only condition (i) of \autoref{def:Veech_fibration}, then the monodromy around the cusps will be \textit{quasi-unipotent}.  
This follows from standard theory (see Schmid \cite[Thm~6.1]{Schmid}), but it can also be seen directly.
As discussed in Section 2, each cusp generator of $C$ has a power which is a multitwist around a set of disjoint curves on $X$, and this multitwist acts on $H_1(X;\Z)$ by a unipotent element \cite[Prop~6.3]{farb_margalit}.
So if $B \rightarrow C$ satisfies condition (i), one can always pull back $\X \rightarrow B$ along a further cover $B' \rightarrow B$ to obtain a family $\X' \rightarrow B'$ with a semistable completion.

\subsubsection{Relation with Lefschetz fibrations}

We note that a Veech fibration $\Xt \rightarrow \overline{B}_m$ is in fact a Lefschetz fibration with connected fibers.  
It is \emph{relatively minimal}, meaning that no fiber contains a $-1$-sphere.  
However, the map $\Xt \rightarrow \overline{B}$ is not injective on the set of critical points, meaning that each singular fiber will generally contain multiple vanishing cycles.
(If we view $\Xt$ as a smooth 4-manifold rather than a complex surface, we could perturb the map to be injective on critical points.)

In the stable family $\Xb$ the vanishing cycles in the singular fiber over a cusp are the core curves of the associated cylinder decomposition.
The monodromy around a cusp in $\Xb$ is a multitwist comprised of powers of Dehn twists about the core curves.
When we pass from the stable family $\Xb$ to the semistable family $\Xt$, we replace each vanishing cycle with a set of parallel vanishing cycles, making the monodromy a single Dehn twist around each curve.

\subsection{Congruence Veech fibrations}\label{subsec:congruence_Veech_fibrations}

\subsubsection{Defining congruence Veech fibrations}
As mentioned in the Introduction, we will focus on a particular collection of torsion free finite-index subgroups of $\Aff^+(X, \omega)$ in construction Veech fibrations: the \textit{level $m$ congruence Veech subgroups}.

Let $(X,\omega)$ be a Veech surface with Teichm\"uller curve $C \rightarrow \M_g$.
Recall that the level $m$ congruence subgroup
\begin{equation*}
    \Mod(X)[m] \coloneqq \Ker(\Mod(X) \rightarrow \Sp(H_1(X;\Z/m\Z)).
\end{equation*}
is torsion free for $m \geq 3$ \cite[Thm~6.9]{farb_margalit}.
We write 
\begin{equation*}
    \Aff^+(X,\omega)[m] \coloneqq \Aff^+(X,\omega) \cap \Mod(X)[m]
\end{equation*}
and $\SL(X,\omega)[m]$ for the image of $\Aff^+(X,\omega)[m]$ in $\SL(X,\omega)$.
Note that in fact
\begin{equation*}
    \Aff^+(X,\omega)[m] \cong \SL(X,\omega)[m],
\end{equation*}
since $\Aut(X,\omega) \cap \Aff^+(X,\omega)[m]$ is trivial (the former is finite while the latter is torsion free).
We define $B_m \coloneqq \H/\SL(X,\omega)[m]$, and we call $B_m \rightarrow C$ the \emph{level $m$ congruence cover} of $C$.

While the cusps of $B_m$ needn't have unipotent monodromy (see after \autoref{def:Veech_fibration}), when they do we define: 
\begin{definition}
    Let $(X,\omega)$ be a Veech surface with Teichm\"uller curve $C \rightarrow \M_g$, and let $B_m \rightarrow C$ be the level $m \geq 3$ congruence cover.
    \textit{If} the cusps of $B_m$ have unipotent monodromy, we obtain a Veech fibration $\Xt_m \rightarrow \overline{B}_m$ with fiber $(X,\omega)$ that we call the \emph{level $m$ congruence Veech fibration} for $(X,\omega)$.
\end{definition}

\subsubsection{Topology of the base}\label{subsec:topology-of-base}
If we let $\rho_m:\Mod(X) \rightarrow \Sp(H_1(X;\Z/m\Z))$ denote the action on homology, then the short exact sequence
\begin{equation*}
    1 \rightarrow \Mod(X)[m] \rightarrow \Mod(X) \xrightarrow{\rho_m} \Sp(H_1(X;\Z/m\Z)) \rightarrow 1,
\end{equation*}
restricts to a sequence
\begin{equation*}
    1 \rightarrow \Aff^+(X,\omega)[m] \rightarrow \Aff^+(X,\omega) \rightarrow \rho_m(\Aff^+(X,\omega)) \rightarrow 1.
\end{equation*}
Let $\PSL(X,\omega)[m]$ denote the image of $\SL(X,\omega)[m]$ in $\PSL(X,\omega)$.
Note that the projection $\SL(X,\omega)[m] \rightarrow \PSL(X,\omega)[m]$ is injective, since $\SL(X,\omega)[m]$ is torsion free and so $-I \notin \SL(X, \omega)[m]$.

There is a well-defined map $\SL(X, \omega) \to \rho_m(\Aff^+(X, \omega))/\rho_m(\Aut(X, \omega))$ and a resulting commutative diagram of exact sequences:
\begin{equation*}
    \begin{tikzcd}
	{\Aff^+(X,\omega)[m]} & {\Aff^+(X,\omega)} & {\rho_m(\Aff^+(X,\omega))} \\
	{\SL(X,\omega)[m]} & {\SL(X,\omega)} & {\rho_m(\Aff^+(X,\omega))/\rho_m(\Aut(X,\omega))} \\
	{\PSL(X,\omega)[m]} & {\PSL(X,\omega)} & {\rho_m(\Aff^+(X,\omega))/\la \rho_m(\Aut(X,\omega)), J\ra}
	\arrow[hook, from=1-1, to=1-2]
	\arrow[hook, from=2-1, to=2-2]
	\arrow[hook, from=3-1, to=3-2]
	\arrow[two heads, from=1-2, to=1-3]
	\arrow[two heads, from=2-2, to=2-3]
	\arrow[two heads, from=3-2, to=3-3]
	\arrow["\cong", from=1-1, to=2-1]
	\arrow["\cong", from=2-1, to=3-1]
	\arrow[from=1-2, to=2-2]
	\arrow[from=2-2, to=3-2]
	\arrow[from=1-3, to=2-3]
	\arrow[from=2-3, to=3-3]
\end{tikzcd}
\end{equation*}
where $J = -I$ if $-I \in \SL(X,\omega)$ and $J=I$ if $-I \not\in \SL(X,\omega)$.
The bottom row corresponds to the Galois cover of hyperbolic orbifolds $B_m \rightarrow \widehat{C} = \H/\PSL(X,\omega)$ with deck group 
\begin{equation*}
    G \coloneqq \rho_m(\Aff^+(X,\omega))/\la \rho_m(\Aut(X,\omega)), J \ra.
\end{equation*}

It follows that to compute the degree of $B_m \rightarrow \widehat{C}$, we must compute the order of the group $G$. 
We carry this out for algebraically primitive $(X,\omega)$ and certain prime $m$ in \autoref{sec:alg_prim_Tcurves}.
Note that in that in these cases, $\Aut(X,\omega)$ is trivial, and the map $\rho_m$ factors through $\SL(X,\omega)$.
We then compute the genus of $B_m$ from the Riemann-Hurwitz formula.
As for the cusps of $B_m$, let $C_1, \ldots, C_n \in \PSL(X,\omega)$ be the cusp generators of $\widehat{C}$.
Then if $C_i$ has order $k_i$ in $G$, there are $\v G \v/k_i$ cusps on $B_m$ lying over $C_i$, each of which is generated by a conjugate of $C_i^{k_i}$.

\section{Topological invariants}\label{sec:top_invariants}
In this section, we discuss the basic topological invariants of Veech fibrations.
Throughout this section, we fix a genus $g$ Veech surface $(X,\omega)$ with Teichm\"uller curve $C$, and let $\Xt \rightarrow B$ be a Veech fibration with fiber $(X,\omega)$.

In \autoref{subsec:eul_char_signature}, we present formulas for the Euler characteristic $e(\Xt)$, the signature $\sigma(\Xt)$, and the Chern class $c_1^2(\Xt)$.
In \autoref{subsec:pi_1}, we discuss $\pi_1(\Xt)$.  
We will see in particular that the map $\pi_1(\Xt) \rightarrow \pi_1(\overline{B})$ is surjective; we provide a sufficient condition for when the map is an isomorphism, and an example where it is not.
In \autoref{subsec:betti_numbers}, we discuss $H_1(\Xt)$.
Again, we have that the map $H_1(\Xt) \rightarrow H_1(\overline{B})$ is surjective; we provide a sufficient condition for when the map is an isomorphism, and an example where it is not.
Finally, in \autoref{subsec:intersection_form}, we discuss the intersection form on $H_2(\Xt)$.

\subsection{Euler characteristic and signature}\label{subsec:eul_char_signature}
\subsubsection{Total twisting}
Both the Euler characteristic and signature formulas involve a quantity we call the \emph{total twisting} of $\Xt \rightarrow B$, defined as follows.
Let $\Delta$ denote the set of cusps of $B$.
Each cusp $c \in \Delta$ has an associated cylinder decomposition with core curves $\gamma_1^c, \ldots, \gamma_{\ell_c}^c$, and the monodromy around $c$ is a $k_i^c$-fold Dehn twist around the curve $\gamma_i^c$ for some $k_1^c, \ldots, k^c_{\ell_c} \in \Z$.
We define the total twisting to be the quantity
\begin{equation*}
    T \coloneqq \sum_{c \in \Delta}\sum_{i=1}^{\ell_c} k_i^c.
\end{equation*}

The significance of the total twisting $T$ is that it is also the total number of vanishing cycles across all the singular fibers of $\Xt$.
Indeed, for each $c \in \Delta$, the semistable curve over $c$ is obtained by contracting $k_i^c$ parallel copies of the curve $\gamma^c_i$ for $1 \leq i \leq \ell_c$.

\subsubsection{Euler characteristic}
We now recount a well-known formula for the Euler characteristic $e(\Xt)$.
\begin{proposition}\label{prop:euler_char}
The Euler characteristic of $\Xt$ is
\[
    e(\Xt) = \chi(\overline{B})\chi(X) + T = 4(g-1)(b-1) + T.
\]
\end{proposition}
\begin{proof}
    By \cite[Prop~III.11.4]{Barth}, we have that
    \begin{equation*}
        e(\Xt) = \chi(\overline{B})\chi(X) + \sum_{\text{singular fibers $F$}} (\chi(F) - \chi(X)).
    \end{equation*}
    The proposition follows from the observation that for each singular fiber $F$, the quanitity $\chi(F) - \chi(X)$ is precisely the number of vanishing cycles on $F$.
\end{proof}

\subsubsection{Signature}
Next, we present a formula for the signature $\sigma(\Xt)$.
The formula mainly follows from the computations in \cite{Moller_PCMI}.
As we remark below, the formula also agrees with Smith's formula \cite{Smith} for the signature of a Lefschetz fibration in terms of the Hodge bundle on $\overline{\M}_g$.

\begin{proposition}\label{prop:signature_formula}
The signature of $\Xt$ is
\[
    \sigma(\Xt) = -2\kappa_\mu\chi(B) - \frac{2}{3}T,
\]
where $\chi(B) = 2 - 2b - \v \Delta \v$ is the Euler characteristic of $B$, $\mu = (m_1, \ldots, m_\ell)$ is the partition of $2g-2$ given by the zeros of $\omega$, and
\begin{equation*}
    \kappa_\mu \coloneqq \frac{1}{12}\sum_{i=1}^\ell \frac{m_i(m_i+2)}{m_i+1}.
\end{equation*}
\end{proposition}

\begin{remark}
    Note that if we let $d$ denote the degree of the cover $B \rightarrow C$, then the quantity $\sigma(\Xt)/d$ is independent of the choice of cover.
    Indeed, we have that $\chi(B) = d\chi(C)$, and $T = dT'$ for a quantity $T'$ depending only on $C$ (see the proof of \autoref{prop:BMY_suff_cond}).
\end{remark}
To prove \autoref{prop:signature_formula}, we fix some notation.
Let $f:\Xt \rightarrow B$ be the Veech fibration.
Let $\omega_{\Xt}$ be the canonical sheaf of $\Xt$.
We let $\omega_{\Xt/\overline{B}}$ denote the \emph{relative dualizing sheaf} of $f$.
This is the sheaf on $\Xt$ which restricts to the dualizing sheaf on each fiber (in particular, it restricts to the canonical sheaf on each smooth fiber); see \cite[Ch.~3]{HarrisMorrison} for further background.

\begin{proof}
    As in the proof of \cite[Prop~3.3]{Moller_PCMI}, we have that
    \begin{align*}
        \chi(\O_{\Xt}) &= \deg f_*\omega_{\Xt/\overline{B}} + (g-1)(b-1) \quad \text{and} \\
        c_1^2(\omega_{\Xt}) &= \omega_{\Xt/\overline{B}}^2 + 8(g-1)(b-1),
    \end{align*}
    where $\chi(\O_{\Xt})$ is the holomorphic Euler characteristic.
    (See, e.g., the book of Xiao \cite{Xiao2008-mg} for more details.)
    From \autoref{prop:euler_char}, the Euler characteristic is
    \begin{equation*}
        c_2(\omega_{\Xt}) = e(\Xt) = 4(g-1)(b-1) + T.
    \end{equation*}
    Using Noether's formula $12\chi(\O_{\Xt}) - c_2(\omega_{\Xt}) = c_1^2(\omega_{\Xt})$, we conclude that
    \begin{equation*}
        12\deg f_*\omega_{\Xt/\overline{B}} - T = \omega_{\Xt/\overline{B}}^2.
    \end{equation*}
    
    In the proof of \cite[Prop~5.12]{Moller_PCMI}, M\"oller computes that
    \begin{equation*}
        \omega_{\Xt/\overline{B}}^2 = -6\kappa_\mu\chi(B).
    \end{equation*}
    By the Hirzebruch signature formula,
    \begin{align*}
        \sigma(\Xt) &= \frac{1}{3}\left(c_1^2(\omega_{\Xt}) - 2c_2(\omega_{\Xt})\right) \\
        &= \frac{1}{3}\left(12\chi(\O_{\Xt}) - 3c_2(\omega_{\Xt})\right) \\
        &= \frac{1}{3}\left(12\deg f_*\omega_{\Xt/\overline{B}} - 3T\right) \\
        &= \frac{1}{3}\left(\omega_{\Xt/\overline{B}}^2 - 2T\right) \\
        &= -2\kappa_\mu\chi(B) - \frac{2}{3}T.
    \end{align*}
\end{proof}

\begin{remark}\label{rmk:smith-formula}
    The formula in \autoref{prop:signature_formula} agrees with Smith's formula \cite{Smith} for the signature of a Lefschetz fibration.
    Letting $\lambda$ denote the first Chern class of the Hodge bundle over $\overline{\M}_g$, Smith's formula says that
    \begin{equation*}
        \sigma(\Xt) = 4(\overline{B} \cdot \lambda) - T.
    \end{equation*}
    (Smith assumes that the map $\Xt \rightarrow \overline{B}$ is injective on the set of critical points, but that makes no difference here).
    Since $\deg f_*\omega_{\Xt/\overline{B}} = \overline{B} \cdot \lambda$, this is precisely the equation
    \begin{equation*}
        \sigma(\Xt) = \frac{1}{3}\left(12\deg f_*\omega_{\Xt/\overline{B}} - 3T\right)
    \end{equation*}
    that we obtain in the course of proving \autoref{prop:signature_formula}.
    
    We also note that M\"oller shows in the proof of \cite[Prop~3.3]{Moller_PCMI} that
    \begin{equation*}
        T = \overline{B} \cdot \delta,
    \end{equation*}
    where $\delta$ is the total boundary divisor of $\overline{\M}_g$.
    Since the vanishing cycles are all non-separating, $\overline{B}$ only intersects the boundary divisor in strata corresponding to irreducible nodal curves of genus $g-1$.
\end{remark}

Observe that in the proof of \autoref{prop:signature_formula}, we obtain the following formula for $c_1^2(\Xt)$.

\begin{lemma}\label{lem:c1_squared}
    We have that
    \begin{equation*}
        c_1^2(\Xt) = -6\kappa_\mu\chi(B) + 8(g-1)(b-1).
    \end{equation*}
\end{lemma}

In Section 5, we use this formula for $c_1^2(\Xt)$ to determine the Kodaira dimension of $\Xt$.
Note that it does not depend on the twisting $T$.

\subsection{Fundamental group}\label{subsec:pi_1}

\subsubsection{Fundamental group as a semidirect prouct}
We start with a formula for $\pi_1(\Xt)$ as a semidirect product of a quotient of $\pi_1(X)$ and $\pi_1(\overline{B})$.
In principle this can be used to derive a presentation for $\pi_1(\Xt)$ in any given example.
The following proposition follows from more general results on Lefschetz fibrations, but we recount the details carefully.

\begin{proposition}\label{prop:pi_1_total_space}
    Let $K$ be the normal closure in $\pi_1(X)$ of all core curves of all cylinder decompositions of $X$.
    There is an isomorphism
    \begin{equation*}
        \pi_1(\Xt) \cong (\pi_1(X)/K) \rtimes \pi_1(\overline{B}).
    \end{equation*}
    Here the action of $\pi_1(\overline{B})$ on $\pi_1(X)/K$ is induced by a map $\pi_1(B) \rightarrow \Aut(\pi_1(X))$ obtained by lifting the monodromy $\pi_1(B) \rightarrow \Out(\pi_1(X))$ of the surface bundle $\X \rightarrow B$.
\end{proposition}
\begin{proof}
    Since $\X \rightarrow B$ is a surface bundle, we have a short exact sequence
    \begin{equation*}
        1 \rightarrow \pi_1(X) \rightarrow \pi_1(\X) \rightarrow \pi_1(B) \rightarrow 1.
    \end{equation*}
    In fact, since $\pi_1(B)$ is free we can choose a splitting
    \begin{equation*}
        \pi_1(\X) \cong \pi_1(X) \rtimes \pi_1(B).
    \end{equation*}
    Here, the action of $\pi_1(B)$ on $\pi_1(X)$ is a lift of the monodromy map $\pi_1(B) \rightarrow \Out(\pi_1(X))$.

    We now add in the relations coming from the singular fibers.
    Assume for convenience that $\Xt$ has a single singular fiber; the general case follows similarly.
    Let $c_1, \ldots, c_k$ be the conjugacy classes of $\pi_1(X)$ corresponding to the vanishing cycles, and let $d$ be the conjugacy class of $\pi_1(B)$ corresponding to a loop around the cusp (when appropriate, we identify $c_i$ and $d$ with their images in $\pi_1(\X)$).
    Van Kampen's theorem gives that $\pi_1(\Xt)$ is the quotient of $\pi_1(\X)$ by the $\pi_1(\X)$-normal closure of the elements $c_1, \ldots, c_k, d$.

    We now make two observations.
    First, observe that the $\pi_1(\X)$-normal closure of $c_1, \ldots, c_k$ is equal to the $\pi_1(X)$-normal closure of the orbits $\pi_1(B) \cdot c_1, \ldots, \pi_1(B) \cdot c_k$.
    That is, if we write $\la\la \cdot \ra\ra_G$ for the $G$-normal closure of a set of elements and consider the groups
    \begin{equation*}
        \widehat{\pi_1(\X)} \coloneqq \pi_1(\X)/\la\la c_1, \ldots, c_k \ra \ra_{\pi_1(\X)}
    \end{equation*}
    and 
    \begin{equation*}
        K \coloneqq \la\la \pi_1(B) \cdot c_1, \ldots, \pi_1(B) \cdot c_k \ra\ra_{\pi_1(X)} \subseteq \pi_1(X),
    \end{equation*}
    then we have
    \begin{equation*}
        \widehat{\pi_1(\X)} \cong \pi_1(X)/K \rtimes \pi_1(B).
    \end{equation*}
    Second, as the monodromy around the cusp of $B$ is a multitwist around the curves $c_1, \ldots, c_k$, the $\pi_1(B)$-normal closure of $d$ acts trivially on $\pi_1(X)/K$.
    
    We conclude that
    \begin{align*}
        \pi_1(\Xt) &\cong \pi_1(\X)/\la\la c_1, \ldots, c_k, d \ra\ra_{\pi_1(\X)} \\
        &\cong \widehat{\pi_1(\X)}/\la\la d \ra\ra_{\widehat{\pi_1(\X)}} \\
        &\cong (\pi_1(X)/K) \rtimes (\pi_1(B)/\la\la d \ra\ra_{\pi_1(B)}) \\
        &\cong (\pi_1(X)/K) \rtimes \pi_1(\overline{B}).
    \end{align*}

    The proposition follows from the fact that the monodromy orbits of the vanishing cycles on $X$ are precisely the core curves of all cylinder decompositions of $X$.
\end{proof}

We emphasize that the splitting of $\pi_1(\Xt)$ obtained from the proof of \autoref{prop:pi_1_total_space} is not natural; it comes from choosing a right inverse of the map $\pi_1(\X) \rightarrow \pi_1(B)$.
In particular, the splitting does not necessarily arise from a section of $\Xt \rightarrow \overline{B}$.

In light of \autoref{prop:pi_1_total_space}, understanding $\pi_1(\Xt)$ amounts to understanding the subgroup $K$.  We are thus led to the following question:

\begin{question}
    When does $K$ equal $\pi_1(X)$?  Or equivalently, when is the map $\pi_1(\Xt) \rightarrow \pi_1(\overline{B})$ an isomorphism?
\end{question}

\subsubsection{A sufficient condition}
We now provide a sufficient condition to determine that $K = \pi_1(X)$, i.e.\ that $\pi_1(\Xt) \cong \pi_1(\overline{B})$.  
In \autoref{sec:alg_prim_Tcurves}, we apply this criterion to all the known families of algebraically primitive Veech surfaces.

\begin{proposition}\label{prop:pi_1_suff_cond}
    Suppose $(X,\omega)$ is a Thurston-Veech surface with a connected polygon representative $P$.
    Suppose that each core curve of the horizontal and vertical cylinders crosses the boundary of $P$ a single time.
    Then $K = \pi_1(X)$ and hence $\pi_1(\Xt) \cong \pi_1(\overline{B})$.
\end{proposition}
\begin{proof}
    Let $\Gamma$ be the union of the core curves of the horizontal and vertical cylinders, viewed as an embedded graph with vertices at the intersection points.
    Since the core curves fill $X$, the graph $\Gamma$ is in fact a 1-skeleton for a cell structure on $X$.
    Let $T \subseteq \Gamma$ be the maximal spanning tree consisting of edges that do not cross the boundary of $P$.
    The hypothesis implies that $\Gamma/T$ is a wedge of circles where each circle is the image of a core curve.
    So the core curves generate $\pi_1(\Gamma)$, and the Proposition follows since the map $\pi_1(\Gamma) \rightarrow \pi_1(X)$ is surjective.
\end{proof}

See \autoref{fig:l-table} for an example and \autoref{fig:EW} for a non-example of a surface satisfying the hypothesis of \autoref{prop:pi_1_suff_cond}.
We now discuss the nonexample.

\subsubsection{A negative example: the Eierlegende Wollmilchsau}\label{subsec:wollmilchsau}
Let $(X,\omega)$ be the \emph{Eierlegende Wollmilchsau}.  
This is the genus 3 square-tiled surface pictured in Figure \ref{fig:EW}.
It is the $Q_8$-cover of the square torus obtained by mapping the generators of $\pi_1(T^2)$ to $i$ and $j$.

\begin{figure}[ht]
    \centering
    \includegraphics[scale=.4]{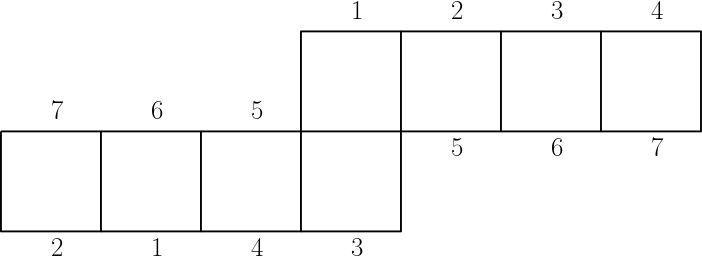}
    \caption{The Eierlegende Wollmilschau.  Side identifcations are given by numbers; unlabeled sides are identified with their opposite.}
    \label{fig:EW}
\end{figure}

We claim that the map $K \rightarrow H_1(X)$ is not surjective, so that $K \neq \pi_1(X)$; we follow an argument of Matheus \cite{matheus--math-overflow}.
We first recall that $\SL(X,\omega) = \SL(2,\Z)$ \cite{Herrlich_Schmithusem_EW}.
This means that the Teichm\"uller curve of $X$ has a single cusp, or equivalently, all cylinder decompositions lie in a single $\Aff^+(X,\omega)$-orbit.
From the natural horizontal cylinder decomposition pictured in Figure \ref{fig:EW}, we see that there are two cylinders, and their core curves are homologous.
Letting $\gamma \in H_1(X;\Z)$ denote one of these core curves (say oriented to the right), we deduce that the image of $K$ in $H_1(X;\Z)$ is precisely the orbit $\Aff^+(X,\omega) \cdot \gamma$.

We claim that, over $\Q$, this orbit is in fact the 2-dimensional tautological subspace $H_1^{st}(X;\Q)$.
The orbit is at least 2-dimensional since the homology class of a vertical core curve intersects $\gamma$ nontrivially.
On the other hand, we can show that the orbit is contained in $H_1^{st}(X;\Q)$.
Let $\sigma \in H_1(X;\Q)$ be the union of the bottom edges of the 8 squares of $(X,\omega)$, oriented to the right, and let $\zeta \in H_1(X;\Q)$ be the union of the left vertical edges, oriented upwards.
The set $\{\sigma,\zeta\}$ is a well-known basis of $H_1^{st}(X;\Q)$.
We see that $2\gamma = \sigma$, so $\gamma \in H_1^{st}(X;\Q)$ and hence $\Aff^+(X,\omega) \cdot \gamma \subseteq H_1^{st}(X;\Q)$.

\subsection{Betti numbers}\label{subsec:betti_numbers}

\subsubsection{First homology}
The semidirect product decomposition in \autoref{prop:pi_1_total_space} will allow us to compute $H_1(\Xt;\Z)$.
By Poincar\'e duality, the rank of $H_1(\Xt;\Z)$ and the Euler characterstic $e(\Xt)$ is enough information to compute all the Betti numbers of $\Xt$.

By \autoref{prop:pi_1_total_space}, we have an isomorphism
\begin{equation*}
    \pi_1(\Xt) \cong \pi_1(X)/K \rtimes \pi_1(\overline{B}),
\end{equation*}
where $K$ is the normal closure of the core curves of all cylinder decompositions of $X$.
From the 5-term exact sequence associated to the Hochschild-Serre spectral sequence (see, e.g., \cite[Cor~VII.6.4]{Brown}) we get a splitting
\begin{equation}\label{eqn:H1(Xt)_splits}
    H_1(\Xt; R) \cong H_1(\pi_1(X)/K;R)_{\pi_1(\overline{B})} \oplus H_1(\overline{B};R)
\end{equation}
for $R = \Z$ or $\Q$, where the subscript denotes coinvariants.

We naturally ask:
\begin{question}\label{ques:H1}
    When does $H_1(\pi_1(X)/K;R)_{\pi_1(\overline{B})}$ vanish?
    Equivalently, when is the natural map $H_1(\Xt;R) \rightarrow H_1(\overline{B};R)$ an isomorphism?
\end{question}

\subsubsection{Sufficient conditions for vanishing coinvariants}
Towards \autoref{ques:H1}, we define the \emph{core curve subspace}
\begin{equation*}
    H_1^{\cc}(X;R) \coloneqq \Im(K \rightarrow H_1(X;R)) = \Im(H_1(K; R) \rightarrow H_1(X; R)).
\end{equation*}
From the short exact sequence
\begin{equation*}
    1 \rightarrow K \rightarrow \pi_1(X) \rightarrow \pi_1(X)/K \rightarrow 1
\end{equation*}
we abelianize to obtain a short exact sequence
\begin{equation*}
    0 \rightarrow H_1^{\cc}(X;R) \rightarrow H_1(X;R) \rightarrow H_1(\pi_1(X)/K;R) \rightarrow 0
\end{equation*}
for $R = \Z$ or $R = \Q$.  In particular,
\begin{equation}\label{eqn:H1(X)_splits}
    H_1(X;\Q) \cong H_1^{\cc}(X;\Q) \oplus H_1(\pi_1(X)/K;\Q).
\end{equation}

We immediately deduce:
\begin{observation}\label{obs:H1_suff_condition}
    If $H_1^{\cc}(X;R) = H_1(X;R)$, then $H_1(\pi_1(X)/K;R) = 0$ and
    \begin{equation*}
        H_1(\Xt;R) \cong H_1(\overline{B};R).
    \end{equation*}
\end{observation}

In Section 6, we'll see that this is enough to determine the Betti numbers of $\Xt$ whenever $(X,\omega)$ is algebraically primitive.

\subsubsection{Example of non-vanishing coinvariants: the Eierlegende Wollmilchsau}
Let $(X,\omega)$ be the Eierlegende Wollmilchsau, as defined above.
We claim that for a particular choice of base $B$ (or equivalently a particular choice of torsion free finite index subgroup of $\Aff^+(X,\omega)$), one has that
\begin{equation*}
    H_1(\pi_1(X)/K;\Q)_{\pi_1(\overline{B})} \neq 0,
\end{equation*}
and hence $b_1(\Xt) > b_1(\overline{B})$.  This follows from two facts:
\begin{enumerate}
    \item For sufficiently large $m$, the action of $\Aff^+(X,\omega)[m]$ on $H_1^{(0)}(X,\Q)$ is trivial.
    \item With respect to the splitting (\ref{eqn:H1(X)_splits}), the intersection $H_1^{(0)}(X;\Q) \cap H_1(\pi_1(X)/K;\Q)$ is nonzero.
\end{enumerate}
Indeed, from these two fact, we get that
\begin{align*}
    0 &\neq H_1^{(0)}(X;\Q) \cap H_1(\pi_1(X)/K;\Q) \\ 
    &\subseteq H_1(\pi_1(X)/K;\Q)^{\Aff^+(X,\omega)[m]} \\ 
    &\cong H_1(\pi_1(X)/K;\Q)_{\Aff^+(X,\omega)[m]}.
\end{align*}
Then for any base $B$ with $\pi_1(B) \subseteq \Aff^+(X,\omega)[m]$, we conclude that $b_1(\Xt) > b_1(\overline{B})$.

Fact (1) follows from the fact that $\Aff^+(X,\omega)$ acts on $H_1^{(0)}(X;\Z)$ via a \textit{finite} group (see \cite[\S~8.5]{forni_matheus_survey}).
Note that $H_1^{(0)}(X;\Z)$ is a direct summand of $H_1(X;\Z)$ because $H_1^{(0)}(X;\Z) = \Ker(\hol)$ and $\Im(\hol)$ is torsion free.
It follows that we can choose $m$ large enough so that this entire finite group acts nontrivially on the image of $H_1^{(0)}(X;\Z)$ in $H_1(X;\Z/m\Z)$, and hence $\Aff^+(X,\omega)[m]$ acts trivially on $H_1^{(0)}(X;\Z)$.

We can prove fact (2) with a dimension count.
Since $X$ has genus 3, we have that $H_1(X;\Q)$ has dimension 6 and $H_1^{(0)}(X;\Q)$ has dimension 4.
But $H_1(\pi_1(X)/K;\Q)$ also has dimension 4: there is a splitting (\ref{eqn:H1(X)_splits}), but we showed in \autoref{subsec:wollmilchsau} that $H_1^{\cc}(X;\Q)$ has dimension at most 2. 

\subsection{Intersection form}\label{subsec:intersection_form}
Finally, we comment on the intersection form $Q$ of the $4$-manifold $\Xt$.
As $\Xt \rightarrow \overline{B}$ is a Lefschetz fibration, a smooth fiber will be a nontrivial homology class with zero self-intersection, and hence $Q$ is indefinite.
Combining our results on Euler characteristic, signature and Betti numbers, we can compute the rank and signature of $Q$ in many cases.
The essential remaining property $Q$ to determine is its parity.

The following result of Chen and M\"oller \cite[Prop~4.8]{ChenMoller} allows us to compute the self-intersection of certain natural homology classes, and hence show that $Q$ is odd in some examples.
Let $(m_1, \ldots, m_\ell)$ be the partition of $2g-2$ corresponding to the zeros of $\omega$.
Assume that $\pi_1(B) \subseteq \Aff^+(X,\omega)$ fixes the zero of order $m_i$ (we could pass to a cover to make this the case).
The $i$th zero then gives a section of $\Xt \rightarrow \overline{B}$, whose image is a homology class $S_i \in H_2(\Xt;\Z)$.

\begin{proposition}[Chen--M\"oller]\label{prop:zero_section}
    We have
    \begin{equation*}
        S_i^2 = \frac{-\chi(B)}{2(m_i + 1)} = \frac{2 - 2b - \v \Delta \v}{2(m_i + 1)}.
    \end{equation*}
\end{proposition}
Were $S_i^2$ to be odd, we'd have that $Q$ is an odd intersection form.

\begin{remark}\label{rmk:per-points-sections}
    A periodic point, that is, a point $p \in X$ with a finite $\Aff^+(X, \omega)$-orbit, also determines a section of $\Xt \to \overline{B}$ (up to passing to a finite cover).
    Examples include the zeros of $\omega$ and the 6 Weierstrass points of a genus 2 Veech surface.
    It would be interesting to find a similar self-intersection formula in this case.
    See \cite{moller_periodic_2006} for more information on periodic points in the context of families over Teichm\"uller curves.
\end{remark}

\section{Complex-geometric invariants}\label{sec:geometric_invariants}
Let $\Xt \rightarrow \overline{B}$ be a Veech fibration with fiber $(X,\omega)$.
In this section, we study two complex geometric properties of the surface $\Xt$: its Kodaira dimension $\kappa(\Xt)$ and its Bogomolov-Miyaoka-Yau (BMY) inequality.
Recall that the Kodaira dimension of a complex surface can be $-\infty$, $0$, $1$, or $2$, dividing minimal complex surfaces into four broad classes (see \cite{Barth} for details on the classification of minimal complex surfaces).
The BMY inequality says that for a minimal general type surface, $c_1^2 \leq 3c_2$, and that if the equality holds then the surface is a ball quotient.

In \autoref{subsec:kodaira_dim}, we show that if $X$ and $B$ have sufficiently high genus, then $\Xt$ is a minimal complex surface with Kodaira dimension 2, i.e., a general type surface.
We also comment on the low genus cases.
In \autoref{subsec:BMY}, we provide a sufficient condition to show that the BMY inequality is strict for $\Xt$.  We apply this to congruence Veech fibrations $\Xt_m$ for sufficiently high $m$.

\subsection{Kodaira dimension}\label{subsec:kodaira_dim}

\subsubsection{General type Veech fibrations}
We begin by showing that most Veech fibrations are general type surfaces, i.e.\ have $\kappa(\Xt) = 2$.

\begin{proposition}\label{thm:kodaira-dimension-genus-high}
    If $X$ has genus $g \geq 2$ and $B$ has genus $b \geq 1$, then $\Xt$ is a minimal complex surface of general type.
\end{proposition}
\begin{proof}
    Recall from \autoref{subsec:the_general_construction} that $\Xt \rightarrow \overline{B}$ is relatively minimal, i.e.\ no fiber contains a sphere of self-intersection $-1$.
    By \cite[Thm~1.4]{Stipsicz}, it follows that $\Xt$ is minimal.
    Because $\kappa(X) = 1$ and $\kappa(\overline{B}) \ge 0$, we apply Iitaka's theorem (see \cite[III, Theorem 18.4]{Barth}) to find
    \begin{equation*}
        \kappa(\Xt) \ge \kappa(X) + \kappa(\overline{B}) \ge 1.
    \end{equation*}
    We must rule out that $\kappa(\Xt) = 1$, i.e.\ that $\Xt$ is a proper elliptic surface.
    Since $\Xt$ is a minimal complex surface such that $1 \le \kappa(\Xt) \le 2$, we have $c_1^2(\Xt) \ge 0$ with equality if and only if $\kappa(\Xt) = 1$.
    Using our formula for $c_1^2(\Xt)$ in \autoref{lem:c1_squared}, we find
    \begin{align*}
        c_1^2(\Xt) = 6 \kappa_\mu (2b - 2 + \v \Delta \v) + 8(g-1)(b-1) \ge 6 \kappa_\mu \v \Delta\v > 0,
    \end{align*}
    where $\Delta$ is the set of cusps of $B$, and the result follows.
\end{proof}

\subsubsection{Fiber genus one}
An important case not included in \autoref{thm:kodaira-dimension-genus-high} is when when the fiber $X$ has genus $g=1$.
As discussed in \autoref{subsec:intro--elliptic-fibrations}, the surface $\Xt$ in this case is an elliptic modular surface.

Recall that in this case, the level $m$ congruence cover $B_m \rightarrow C$ corresponds to the principal congruence subgroup $\Gamma(m) \subseteq \SL(2,\Z)$.
The base $B_m$ will have genus $b \geq 2$ if $m \geq 7$, genus $b=1$ if $m = 6$, and genus $b=0$ if $m \in \{3,4,5\}$.
A computation of geometric genus shows that $\kappa(\Xt_m) = 1$ for $m \geq 5$.
By \cite[\S 5]{Sebbar_Besrour_2021}, $\Xt_4$ is an elliptic semistable K3 surface with $\kappa(\Xt_4) = 0$ and $\Xt_3$ is a rational semistable Beauville surface with $\kappa(\Xt_3) = -\infty$.
See \cite[\S 5]{Sebbar_Besrour_2021} for more examples with low Kodaira dimension.

\subsubsection{Base genus 0}
Another omission from \autoref{thm:kodaira-dimension-genus-high} is the case of base genus $b=0$.
In these cases, we cannot immediately conclude that the total space $\Xt$ is minimal.
Apart from the $g=1$ cases, an example with $b=0$ is the level 3 congruence fibration of the double pentagon surface.
We give a direct proof of minimality for this example in \autoref{rem:double_pentagon}. 

\subsection{BMY inequality}\label{subsec:BMY}
\subsubsection{Sufficient condition for strict inequality}
As mentioned above, the BMY inequality says that for a minimal general type surface, the Chern numbers satisfy $c_1^2 \leq 3c_2$, or equivalently the Euler characteristic and signature satisfy $\sigma \leq \frac{1}{3}e$ (see e.g.\ \cite[Chapter VII, Theorem 4.1]{Barth}).
Moreover, if the inequality is an equality, then the surface is a ball quotient.
There are few known families of surfaces for which the equality holds; we naturally ask:

\begin{question}
    Are there any examples of Veech fibrations $\Xt \to \overline{B}$ for which $c_1^2(\Xt) = 3c_2(\Xt)$?  
\end{question}

Note that by \autoref{lem:c1_squared} and \autoref{prop:euler_char}, we may write the BMY inequality for a Veech fibration $\Xt \rightarrow \overline{B}$ as
\begin{equation*}
    -6 \kappa_{\mu} \chi(B) + 8(g-1)(b-1) \le 3\left(T + 4(g-1)(b-1)\right),
\end{equation*}
or equivalently,
\begin{equation*}    
    -6 \kappa_{\mu} \chi(B) \le 3T + 4(g-1)(b-1).
\end{equation*}

We have the following sufficient condition to show that the BMY inequality is strict.

\begin{proposition}\label{prop:BMY}
    Suppose $B$ has genus $b \geq 1$.
    The BMY inequality for $\Xt$ is strict if
    \begin{equation*}
        \frac{g-1}{2}\v \Delta \v < T,
    \end{equation*}
    where $g$ is the genus of $X$, $\Delta$ is the set of cusps on $B$, and $T$ is the total twisting.
\end{proposition}

To prove \autoref{prop:BMY}, we begin with a quick lemma.
Recall that the constant $\kappa_\mu$ is defined as
\begin{equation*}
    \kappa_\mu = \frac{1}{12}\sum_{i=1}^\ell\frac{m_i(m_i+2)}{m_i+1},
\end{equation*}
where $\mu = (m_1, \ldots, m_\ell)$ is the partition of $2g-2$ corresponding to the zeros of $\omega$.
We first bound this parameter from above.

\begin{lemma}\label{lem:kappa_mu_bound}
    For all partitions $\mu$ of $2g - 2$, we have $12\kappa_\mu \le 3g - 3$, with equality if and only if $\mu = (1, 1, \ldots, 1)$.
\end{lemma}
\begin{proof}
    Observe that
    \begin{align*}
        12 \kappa_\mu = \sum \frac{m_i(m_i + 2)}{m_i + 1} = \sum \frac{(m_i + 1)^2 - 1}{m_i + 1} &= \sum (m_i + 1) - \sum \frac{1}{m_i + 1}. \\
        &= (2g - 2) + |\mu| - \sum \frac{1}{m_i + 1}.
    \end{align*}
    We seek to maximize the function $f(\mu) \coloneqq |\mu| - \sum (m_i + 1)^{-1}$ over all partitions of $2g - 2$.
    
    We claim that $f(\mu)$ has a strict maximum at $\mu = (1,1, \ldots, 1)$.
    To see this, consider a partition $\mu = (m_1, \dots, m_\ell)$ having some $m_j > 1$, and construct the partition $\mu' \coloneqq (m_1, \dots, m_j - 1, \dots, m_\ell, 1)$.
    Then we compute
    \begin{align*}
        f(\mu') - f(\mu) &= (\ell + 1) - \ell - \left(\frac{1}{(m_j - 1) + 1} + \frac{1}{1 + 1}\right) + \left(\frac{1}{m_j + 1}\right)\\
        &= \frac{1}{2} - \frac{1}{m_j} + \frac{1}{m_j+1}\\
        &\ge \frac{1}{m_j+1}\\
        &> 0,
    \end{align*}
    where we used that $m_j \ge 2$ in the third line.
    The claim follows inductively.

    When $\mu = (1,1, \ldots, 1)$, we see that
    \begin{align*}
        12\kappa_\mu = \sum_{i = 1}^{2g - 2} \frac{1(1 + 2)}{(1 + 1)} = \frac{3}{2}(2g-2) = 3g - 3,
    \end{align*}
    giving the claimed bound on $12 \kappa_{\mu}$.
\end{proof}

Now, we can prove \autoref{prop:BMY}.

\begin{proof}[Proof of \autoref{prop:BMY}]
    Recall from above that we can write the BMY inequality as
    \begin{equation*}
        -6 \kappa_{\mu} \chi(B) \le 3T + 4(g-1)(b-1).
    \end{equation*}
    By \autoref{lem:kappa_mu_bound}, we have that
    \begin{align*}
        -6\kappa_\mu\chi(B)
        &\leq -\frac{1}{2}(3g-3)\chi(B) \\
        &= \frac{3}{2}(g-1)(2b-2+\v \Delta \v) \\
        &= 3(g-1)(b-1) + \frac{3}{2}(g-1)\v \Delta \v.
    \end{align*}
    The BMY inequality will be strict if
    \begin{equation*}
        3(g-1)(b-1) + \frac{3}{2}(g-1)\v \Delta \v < 3T + 4(g-1)(b-1),
    \end{equation*}
    or equivalently if
    \begin{equation*}
        \frac{3}{2}(g-1)\v \Delta \v < 3T + (g-1)(b-1).
    \end{equation*}
    By assumption, we have that $(g-1)(b-1) \geq 0$, so from here the claim follows.
\end{proof}

\begin{remark}\label{rem:Jeremy's_remark}
    Jeremy Kahn has described to us an alternative sufficient condition: if $\pi_1(\Xt) \isom \pi_1(\overline{B})$, then the BMY inequality is strict.
    The idea is as follows.
    Suppose we had $c_1^2(\Xt) = 3 c_2(\Xt)$, so that $\Xt \cong \C\H^2 / \Gamma$ for some discrete subgroup $\Gamma < \PU(2, 1)$.
    As $\Xt$ is compact, the fundamental theorem of geometric group theory implies that $\Gamma$ is quasi-isometric to $\C\H^2$.
    By assumption $\Gamma \isom \pi_1(\overline{B})$ is a surface group of a compact surface, so that $\Gamma$ is quasi-isometric to $\H^2$ by the same reasoning.
    But it is well known that $\H^2$ is not quasi-isometric to $\C\H^2$, as they are inequivalent symmetric spaces.
\end{remark}

\subsubsection{Strict inequality for congruence fibrations}

We next apply \autoref{prop:BMY} to congruence Veech fibrations.

\begin{proposition}\label{prop:BMY_suff_cond}
    Fix a Veech surface $(X,\omega)$, and let $\Xt_m \rightarrow \overline{B}_m$ be its level $m$ congruence Veech fibration.
    Then, for $m$ sufficiently large, the BMY inequality for $\Xt_m$ is strict.
\end{proposition}
\begin{proof}
    Let $\Delta_m$ denote the set of cusps on $B_m$, and let $T_m$ denote the total twisting on $B_m$.
    By \autoref{prop:BMY}, it's enough to show that
    \begin{equation*}
        \frac{g-1}{2}\v \Delta_m \v < T_m.
    \end{equation*}
    for sufficiently large $m$.
    
    Let $\Delta_C$ denote the set of cusps on the hyperbolic orbifold $\widehat{C}$.
    For each $c \in \Delta_c$, let $\phi_c$ denote the associated minimal multitwist (as defined in \autoref{subsec:Veech_surfaces}).
    By the formula for a multitwist acting on homology (see \cite[Prop~6.3]{farb_margalit}), the order of $\rho_m(\phi_c) \in \Sp(H_1(X;\Z/m\Z))$ grows with $m$.

    Now for each $c \in \Delta_C$, let $A_c \in \PSL(X,\omega) = \pi_1^{orb}(\widehat{C})$ denote the cusp generator.
    Let $\psi_m:\PSL(X,\omega) \rightarrow G_m$ be the monodromy homomorphism of the congruence cover $B_m \rightarrow \widehat{C}$ (as described in \autoref{subsec:congruence_Veech_fibrations}).
    Since the order of $\rho_m(\phi_c)$ grows with $m$, it follows that the order of $\psi_m(A_c) \in G_m$ grows with $m$.

    Let $d$ denote the degree of the congruence cover $B_m \rightarrow \widehat{C}$.
    We see that the number of cusps on $B_m$ lying above $c \in \Delta_c$ is $\frac{d}{\textrm{ord}(\psi_m(A_c))}$, where $\textrm{ord}(\psi_m(A_c))$ denotes the order of $\psi_m(A_c)$ in $G_m$.
    Therefore,
    \begin{equation*}
        \v \Delta_m \v = \sum_{c \in \Delta_C} \frac{d}{\textrm{ord}(\psi_m(A_c))}.
    \end{equation*}
    On the other hand, for each $c \in \Delta_C$, let $T_c$ denote the total number of twists in the minimal multitwist $\psi_m$.
    Each cusp generator $A_c$ corresponds to a $k_c$th root of the multitwist $\phi_c$ for some $k_c \in \Z_{\geq 1}$.
    It follows that the monodromy around each cusp lying over $c$ will be a multitwist comprised of $\textrm{ord}(\psi_m(A_c))\frac{T_c}{k_c}$ twists.
    In other words, we have that
    \begin{align*}
        T_m &= \sum_{c \in \Delta_C} \sum_{\text{cusps over $c$}} \textrm{ord}(\psi_m(A_c))\frac{T_c}{k_c} \\
        &= \sum_{c \in \Delta_C} \left(\frac{d}{\textrm{ord}(\psi_m(A_c))}\right)\textrm{ord}(\psi_m(A_c))\frac{T_c}{k_c} \\
        &= \sum_{c \in \Delta_C} d\frac{T_c}{k_c}.
    \end{align*}

    We see that $\v \Delta_m \v /d$ decreases with $m$, while $T_m/d$ is constant in $m$.
    The desired inequality $\frac{g-1}{2}\v \Delta_m \v < T_m$ thus holds for sufficiently large $m$.
\end{proof}

\section{Algebraically primitive Teichm\"uller curves}\label{sec:alg_prim_Tcurves}
In this section, we restrict our attention to \emph{algebraically primitive} Veech surfaces, i.e.\ Veech surfaces whose trace field $K$ has degree $g$ over $\Q$.
In particular, we apply the results of \autoref{sec:top_invariants} and \autoref{sec:geometric_invariants} to the algebraically primitive case, and explicitly compute invariants of congruence fibrations for all known algebraically primitive Veech surfaces.

In \autoref{subsection:alg_prim_background}, we recount some facts about algebraically primitive Veech surfaces, and in particular discuss real multiplication and the holonomy homomorphism.
In \autoref{subsec:consequences_for_cong_fibs}, we apply the results of \autoref{sec:top_invariants} and \autoref{sec:geometric_invariants} to obtain some general results on the invariants of Veech fibrations for algebraically primitive surfaces.
In \autoref{subsec:action-homology-char-p}, we prove \autoref{thm:action_on_homology_mod_p}, which gives a criterion to compute the degree of the congruence cover for algebraically primitive surfaces.
In \autoref{subsec:TV_construction}, we review the Thurston-Veech construction, which is used to build all known examples of algebraically primitive surfaces.
We leverage this construction to verify that all known algebraically primitive surfaces verify the criterion of \autoref{thm:action_on_homology_mod_p}.
Finally, we compute the invariants of the congruence Veech fibrations for the genus 2 Weierstrass eigenforms, the algebraically primitive regular polygon surfaces, and the sporadic examples in \autoref{subsec:l_tables}, \autoref{subsec:b-m}, and \autoref{subsec:sporadic} respectively.

\subsection{Homology of algebraically primitive Veech surfaces}\label{subsection:alg_prim_background}

\subsubsection{Algebraic and geometric primitivity}
A genus $g$ Veech surface $(X,\omega)$ is called \emph{algebraically primitive} if its trace field $K$ has degree $g$ over $\Q$.
Any algebraically primitive Veech surface is \emph{geometrically primitive}, meaning it is not a translation cover of another Veech surface; the converse is not true in general.
It follows that for $(X,\omega)$ algebraically primitive, the automorphism group $\Aut(X,\omega)$ is trivial (otherwise $X \rightarrow X/\Aut(X,\omega)$ would be a translation cover) and the derivative map $D:\Aff^+(X,\omega) \rightarrow \SL(X,\omega)$ is an isomorphism.
We may then freely identify an affine automorphisms with its linear part.

\subsubsection{Real multiplication}
As explained in \autoref{subsec:action-on-homology}, if $(X,\omega)$ is algebraically primitive, there is a splitting
\begin{equation}\label{eqn:splitting_VHS}
    H_1(X;\R) \cong \bigoplus_{i=1}^g H_1^{st}(X)^{\sigma_i}
\end{equation}
where $\sigma_1, \ldots, \sigma_g$ are the distinct embeddings $K \hookrightarrow \R$.
(Since this splitting is not defined over $\Q$, it generally does not descend to a splitting of $H_1(X;\Z/m\Z)$.)
It follows that the Jacobian of $X$ admits \emph{real multiplication} by the trace field $K$.
For more details on real multiplication, see, for example, \cite{McMullen_Billiards_Tcurves_Hilbert_modular_surfaces} or \cite{McMullen_survey_2022}.
For our purposes, we only need the following formalism:

\begin{lemma}\label{lem:real_multiplication}
    Let $(X,\omega)$ be an algebraically primitive Veech surface of genus $g$, and suppose that $\alpha = \sum_{i=1}^k c_i\tr(D\phi_i)$ for some $c_i \in \Z$ and $\phi_i \in \Aff^+(X,\omega)$.
    Then there exists a ring homomorphism
    \begin{align*}
        \Z[\alpha] &\rightarrow \End(H_1(X;\Z))
    \end{align*}
    sending $\alpha$ to an automorphism $M_\alpha \in \GL(H_1(X;\Z))$ whose dual satisfies $(M_\alpha)^*\omega = \alpha\omega$.
\end{lemma}
\begin{proof}
    We define the endomorphism
    \begin{equation*}
        M_\alpha \coloneqq \sum_{i=1}^k c_i\left((\phi_i)_* + (\phi_i^{-1})_*\right) \in \End(H_1(X;\Z)),
    \end{equation*}
    where $(-)_*$ denotes the induced action of $\Aff^+(X, \omega)$ on $H_1(X;\Z)$.
    It suffices to show that the assignment $\alpha \mapsto M_\alpha$ is well-defined, that is, to show $m_\alpha(M_\alpha) = 0$ where $m_\alpha(x)$ is the minimal polynomial of $\alpha$.
    
    In fact, following \cite[$\S$4]{McMullen_survey_2022} we will show that $M_\alpha$ is diagonalizable and the eigenvalues of $M_\alpha$ are the Galois conjugates of $\alpha$, i.e., the roots of $m_\alpha(x)$.
    It suffices to show this for the dual endomorphism
    \begin{equation*}
        M^\alpha \coloneqq \sum_{i=1}^k c_i\left((\phi_i)^* + (\phi_i^{-1})^*\right) \in \End(H^1(X;\Z))
    \end{equation*}
    where $(-)^*$ denotes the induced action on $H^1(X;\Z)$.
    
    Recall that for any $A = D\phi \in \SL(X,\omega)$, we have
    \begin{equation*}
        \phi^*(\omega) =
        \begin{pmatrix}
            1 & i
        \end{pmatrix}
        A
        \begin{pmatrix}
            \Re(\omega) \\ \Im(\omega)
        \end{pmatrix}.
    \end{equation*}
    Since
    \begin{equation*}
        A_i + A_i^{-1} = 
        \begin{pmatrix}
            \tr(A_i) & 0 \\ 0 & \tr(A_i)
        \end{pmatrix},
    \end{equation*}
    it follows that
    \begin{equation*}
        \left((\phi_i)^* + (\phi_i^{-1})^*\right)\omega = \tr(A_i)\omega,
    \end{equation*}
    and thus 
    $M^\alpha\omega = \alpha\omega$.
    
    Moreover, for any Galois conjugate element $\sigma(\omega) \in H^1(X;\R)$, we get that $M^\alpha\sigma(\omega) = \sigma(\alpha)\sigma(\omega)$.
    Thus we conclude that $M^\alpha$ has the claimed eigenvalues, and so our map $\Z[\alpha] \rightarrow \End(H_1(X;\Z))$ is well defined.
\end{proof}

\subsubsection{The holonomy homomorphism}
The second key component in our study of algebraically primitive surfaces is the \emph{holonomy homomorphism}
\begin{align*}
    \hol:H_1(X;\Z) &\rightarrow \R^2, \qquad
    \gamma \mapsto \int_\gamma \omega.
\end{align*}

The holonomy map satisfies two important equivariance properties:
\begin{enumerate}
    \item The map $\hol$ is $\Aff^+(X,\omega)$-equivariant, where $\Aff^+(X,\omega)$ acts on $\R^2$ via $\SL(X,\omega)$.
    That is, for any $\gamma \in H_1(X;\Z)$ and $\phi \in \Aff^+(X,\omega)$ we have
    \begin{equation*}
        \hol(\phi_*\gamma)
        = \int_{\phi_*\gamma} \omega
        = \int_\gamma \phi^*\omega
        = D\phi \left(\int_ \gamma\omega \right)
        = D\phi (\hol(\gamma)).
    \end{equation*}
    \item The map $\hol$ is $\Z[\alpha]$-equivariant with respect to the $M_\alpha$-action in \autoref{lem:real_multiplication}.
    That is, for any $\gamma \in H_1(X;\Z)$ we have
    \begin{equation*}\label{eqn:alpha_equivariance}
        \hol(M_\alpha\gamma) = \int_{M_\alpha\gamma}\omega
        = \int_{\gamma} (M_\alpha)^*\omega
        = \int_\gamma \alpha\omega
        = \alpha\int_\gamma \omega
        = \alpha\hol(\gamma).
    \end{equation*}
\end{enumerate}

Together these give:
\begin{lemma}\label{lem:hol_inj_and_VG_commutes_RM}
    Suppose $(X,\omega)$ is an algebraically primitive Veech surface.
    Then,
    \begin{enumerate}[label=(\roman*)]
        \item the holonomy homomorphism is injective, and
        \item the action of $\Aff^+(X,\omega)$ on $H_1(X;\Z)$ commutes with the automorphism $M_\alpha$ defined in \autoref{lem:real_multiplication}.
    \end{enumerate}
\end{lemma}
\begin{proof}
    \textbf{(i)} We show that the image of $\hol$ has rank $2g$, the rank of $H_1(X ; \Z)$, from which the claim follows.
    By the Primitive Element Theorem, there exists $\beta \in K$ such that $K = \Q(\beta)$.
    Since $\beta \in K$, we know that $\beta = \sum_{i=1}^k q_i \tr(A_i)$ for $q_i \in \Q$ and $A_i \in \SL(X,\omega)$.
    Up to scaling $\beta$, we may assume that $q_i \in \Z$, so $\beta$ satisfies the hypothesis of \autoref{lem:real_multiplication}.
    Since $(X,\omega)$ is algebraically primitive, we know that $K$ has degree $g$, and hence $\beta$ has degree $g$.
    Let $\gamma, \delta \in H_1(X;\Z)$ be two homology classes such that $\hol(\gamma)$ and $\hol(\delta)$ are not parallel.
    (For instance, take $\gamma$ and $\delta$ to be core curves of cylinders from transverse cylinder decompositions).
    Since $\beta$ has degree $g$, the set of $2g$ holonomy vectors
    \begin{equation*}
        \bigcup_{i = 0}^{g-1} \{\hol(M_\beta^i\gamma),\hol(M_\beta^i\delta)\}
        = \bigcup_{i=0}^{g-1} \{\beta^i \hol(\gamma), \beta^i\hol(\delta)\}
    \end{equation*}
    is linearly independent over $\Z$.

    \textbf{(ii)} We want to show that, for any $\phi \in \Aff^+(X,\omega)$ and any $\gamma \in H_1(X;\Z)$, we have $\phi_*M_\alpha\gamma = M_\alpha \phi_*\gamma$.
    Take any $\alpha$ as in \autoref{lem:real_multiplication}.
    We have
    \begin{equation*}
        \hol(\phi_*M_\alpha\gamma)
        = D\phi (\alpha \hol(\gamma))
        = \alpha D\phi(\hol(\gamma))
        = \hol(M_\alpha \phi_*\gamma),
    \end{equation*}
    and the result now follows from $\textbf{(i)}$.
\end{proof}

\subsubsection{Cylinder bound}

Finally, we record a special case of a theorem of M\"oller \cite[Thm~2.1]{Moller_finiteness_results} on the genus of the stable curves associated to the cusps of a Teichm\"uller curve $C$.
In particular, it gives bounds on the number of cylinders in each cylinder decomposition of an algebraically primitive Veech surface \cite[Cor~2.2]{Moller_finiteness_results}.

\begin{theorem}[M\"oller]\label{thm:Moller_cyl_bound}
    Let $(X,\omega) \in \Omega \mathcal{M}_g(m_1, \dots, m_s)$ be a genus $g$ algebraically primitive Veech surface with Teichm\"uller curve $C$.
    Let $F$ be the stable curve associated to a cusp of $C$.
    Then, each component of $F$ has genus 0.
    In particular, every cylinder decomposition of $(X,\omega)$ has at least $g$ and at most $g + s - 1$ cylinders.
\end{theorem}

\subsection{Consequences for congruence Veech fibrations}\label{subsec:consequences_for_cong_fibs}

Throughout this section, fix a genus $g$ algebraically primitive Veech surface $(X,\omega)$ with Teichm\"uller curve $C$.

\subsubsection{Unipotent monodromy}
Let $B_m \rightarrow C$ be the level $m \geq 3$ congruence cover.  
Recall from \autoref{subsec:Defining_Veech_fibrations} that we obtain a level $m$ congruence Veech fibration $\Xt_m \rightarrow B_m$ if and only if the monodromy around each cusp of $B_m$ is unipotent.
We can verify that this holds automatically when $(X,\omega)$ is algebraically primitive.

\begin{lemma}\label{lem:alg-prim-unipotent}
    Let $(X,\omega)$ be an algebraically primitive Veech surface with Teichm\"uller curve $C$, and let $B_m \rightarrow C$ be the level $m \geq 3$ congruence cover. 
    Then, the monodromy around each cusp of $B_m$ is unipotent.
\end{lemma}
\begin{proof}
    Let $\rho:\Aff^+(X,\omega) \rightarrow \Sp(H_1(X;\Z))$ and $\rho_m:\Aff^+(X,\omega) \rightarrow \Sp(H_1(X;\Z/m\Z))$ denote the action on homology.
    Recall that $\Aut(X,\omega)$ is trivial, so $\rho$ and $\rho_m$ factor through maps $\widehat{\rho}:\SL(X,\omega) \rightarrow \Sp(H_1(X;\Z))$ and $\widehat{\rho}_m:\SL(X,\omega) \rightarrow \Sp(H_1(X;\Z/m\Z))$.
    We view $\pi^\text{orb}_1(C)$ and $\pi_1(B_m)$ as $\SL(X,\omega)$ and $\Ker(\widehat{\rho}_m) \subseteq \SL(X,\omega)$ respectively.

    The decomposition (\ref{eqn:splitting_VHS}) and formula (\ref{eqn:tautological_action}) imply that if $U \in \SL(X,\omega)$ is unipotent, then $\rho(U) \in \Sp(H_1(X;\Z))$ is unipotent.
    It therefore suffices to show that the stabilizer of each cusp of $B_m$ is generated by a unipotent matrix.
    In other words, let $P \subseteq \SL(X,\omega)$ be a parabolic subgroup whose conjugacy class corresponds to a cusp of $C$, and let $P_m = P \cap \Ker(\widehat{\rho}_m)$.
    Then, we must show that that $P_m$ is generated by a unipotent matrix.

    Recall from \autoref{subsec:teichmuller_curves} that either $P = \la U \ra$, $P = \la U, -U \ra = \la U, -I \ra$ or $P = \la -U \ra$, where $U \in \SL(2,\R)$ is some unipotent matrix.
    It follows that $P_m$ is generated by $U^k$ or $-U^k$ for some $k \in \Z$.  However, we claim that $P_m$ cannot be generated by $-U^k$ since $-U^k \not\in \Ker(\widehat{\rho}_m)$.  
    Indeed, from the decomposition (\ref{eqn:splitting_VHS}) and the formula (\ref{eqn:tautological_action}) we see that $\widehat{\rho}_m(-U^k) = -U'$ for some unipotent $U' \in \Sp(H_1(X;\Z/m\Z))$.
    In particular, $\widehat{\rho}_m(-U^k)$ is nontrivial.
    
    We conclude that $P_m = \la U^k \ra$ for some $k \in \Z$, and $P_m$ is generated by a unipotent matrix as desired.
\end{proof}

\subsubsection{Homology}
Let $\Xt_m \rightarrow B_m$ be a level $m \geq 3$ congruence Veech fibration.
Recall from \autoref{subsec:betti_numbers} that we have a splitting
\begin{equation*}
    H_1(\Xt_m;\Q) = H_1(\pi_1(X)/K;\Q)_{\pi_1(\overline{B}_m)} \oplus H_1(\overline{B}_m;\Q),
\end{equation*}
where $K \subseteq \pi_1(X)$ is the normal closure of all core curves of all cylinder decompositions of $(X,\omega)$.
In fact, the coinvariants term vanishes when $(X,\omega)$ is algebraically primitive.

\begin{proposition}
    If $(X,\omega)$ is algebraically primitive, then for any $m \geq 3$, 
    \begin{equation*}
        H_1(\Xt_m;\Q) \cong H_1(\overline{B}_m;\Q).
    \end{equation*}
\end{proposition}
\begin{proof}
    By \autoref{obs:H1_suff_condition}, it's enough to show that the cylinder core curves of $(X,\omega)$ span $H_1(X;\Q)$.
    This follows from \autoref{thm:Moller_cyl_bound}.
    
    Indeed, assume without loss of generality that $(X,\omega)$ is Thurston-Veech.
    Let $V,W \subseteq H_1(X;\Q)$ be the subspaces spanned by the horizontal and vertical core curves respectively.
    Both $V$ and $W$ are isotropic, and by \autoref{thm:Moller_cyl_bound}, they both have dimension at least $g$.
    Since each horizontal core curve has a nonzero algebraic intersection with some vertical core curve, and vice versa, we have that $V \cap W = 0$.
    It follows that $H_1(X;\Q) = V \oplus W$, and in particular the cylinder core curves span $H_1(X;\Q)$.
\end{proof}

\subsubsection{BMY-inequality}
Finally, we show that for an algebraically primitive Veech surface, the congruence Veech fibrations always have a strict BMY-inequality.

\begin{proposition}\label{prop:alg_prim_BMY}
    Let $(X,\omega)$ be an algebraically primitive Veech surface, and let $\Xt_m \rightarrow B_m$ be the level $m \geq 3$ congruence Veech fibration with fiber $(X,\omega)$.
    Suppose $B_m$ has genus $b \geq 1$.
    Then, the BMY-inequality is strict, i.e.\ 
    \begin{equation*}
        c_1^2(\Xt_m) < 3c_2(\Xt_m).
    \end{equation*}
\end{proposition}
\begin{proof}
    Following \autoref{prop:BMY}, it's enough to show that
    \begin{equation*}
        \frac{g-1}{2}\v \Delta \v < T,
    \end{equation*}
    where $\Delta$ is the set of cusps of $B_m$ and $T$ is the total twisting.
    By \autoref{thm:Moller_cyl_bound}, each cylinder decomposition has at least $g$ cylinders.
    By \autoref{lem:alg-prim-unipotent}, the monodromy around each cusp of $B_m$ is a positive power of a Dehn twist around each cylinder.
    We conclude that $T \geq g\v \Delta \v$, giving the desired inequality.
\end{proof}

\subsection{The action on homology in positive characteristic}\label{subsec:action-homology-char-p}

Given an algebraically primitive Veech surface $(X,\omega)$, our goal now is to compute the image of $\rho_m:\Aff^+(X,\omega) \rightarrow \Sp(H_1(X;\F_p))$ for an explicit, infinite set of primes $p$.
In particular, we provide a sufficient condition to compute the image.
Later, we will check that this condition holds for the genus 2 Weierstrass eigenforms in the minimal stratum, the regular polygon surfaces, and the sporadic examples.

We have the following criterion to compute $\Im(\rho_m)$.
\begin{theorem}\label{thm:action_on_homology_mod_p}
Let $(X,\omega)$ be a genus $g$ algebraically primitive Thurston-Veech surface with trace field $K$.
Let the horizontal and vertical multitwists $\phi_H, \phi_V \in \Aff^+(X, \omega)$ have derivatives
\begin{equation*}
D\phi_H = \begin{pmatrix}
1 & c \\ 0 & 1
\end{pmatrix},\ 
D\phi_V = \begin{pmatrix}
1 & 0 \\ d& 1
\end{pmatrix} \in \SL(X,\omega).
\end{equation*}
Let $\alpha \coloneqq cd = \tr(D\phi_H D\phi_V) - \tr(I)$, and suppose $\Im(\hol)$ is a free $\Z[\alpha]$-module of rank $2$ with a $\Z[\alpha]$-basis of the form $\{\hol(\gamma) \coloneqq \left( \begin{smallmatrix} x \\ 0 \end{smallmatrix} \right),\hol(\delta) \coloneqq \left( \begin{smallmatrix} 0 \\ y \end{smallmatrix} \right)\}$.
Let $p \geq 3$ be a prime such that the minimal polynomial $m_\alpha(x)$ of $\alpha$ is irreducible over $\F_p$ and $p^g \neq 9$.  Then,
\begin{equation*}
\Im\left(\Aff^+(X,\omega) \xrightarrow{\rho_m} \Sp(H_1(X;\F_p))\right) \cong \SL(2,\F_{p^g}).
\end{equation*}
\end{theorem}
\begin{proof}
We start by showing that $\rho_m$ factors through $\SL(2,\F_{p^g})$.
By part (i) of \autoref{lem:hol_inj_and_VG_commutes_RM} and our hypothesis on $\Im(\hol)$, we can view $H_1(X;\Z)$ as a free $\Z[\alpha]$-module of rank $2$ with basis $\{\gamma, \delta\}$, where $\alpha$ acts by the automorphism $M_\alpha$.
Moreover, part (ii) of \autoref{lem:hol_inj_and_VG_commutes_RM} implies that $\Aff^+(X,\omega)$ acts on $H_1(X;\Z)$ by $\Z[\alpha]$-linear automorphisms.
Thus the image of $\Aff^+(X,\omega)$ in $\Aut^+_{\Z}(H_1(X;\Z)) = \SL(H_1(X;\Z))$ lies in the subgroup $\Aut^+_{\Z[\alpha]}(H_1(X;\Z))$, which by our hypothesis on $\Im(\hol)$ is isomorphic to $\SL(2,\Z[\alpha])$.

There is a projection
\begin{equation*}
\Z[\alpha] = \Z[x]/(m_\alpha(x)) \rightarrow \F_p[x]/(m_\alpha(x)) \cong \F_{p^g}.
\end{equation*}
Writing $\overline{\alpha}$ for the image of $\alpha$ under this projection (so $\F_p[x]/(m_\alpha(x)) = \F_p(\overline{\alpha}) \cong \F_{p^g}$), we have that $H_1(X;\F_p)$ is an $\F_p(\overline{\alpha})$-vector space of dimension 2.
Moreover, $\Aff^+(X,\omega)$ acts on $H_1(X;\F_p)$ by $\F_p(\overline{\alpha})$-linear automorphisms.
Thus, the image of $\Aff^+(X,\omega)$ in $\Aut^+_{\F_p}(H_1(X;\F_p)) = \SL(H_1(X;\F_p))$ lies in the subgroup
\[
    \Aut^+_{\F_p(\overline{\alpha})}(H_1(X;\F_p)) \cong \SL(2,\F_{p}(\overline{\alpha})) \cong \SL(2,\F_{p^g}).
\]
This gives the desired factorization.

We now show that the map $\Aff^+(X,\omega) \rightarrow \Aut^+_{\F_p(\overline{\alpha})}(H_1(X;\F_p))$ is surjective, from which the theorem follows.
The $\Z[\alpha]$-basis $\{\gamma, \delta\}$ of $H_1(X; \Z)$ gives us a particular identification $\Aut^+_{\Z[\alpha]}(H_1(X;\Z)) \cong \SL(2,\Z[\alpha])$, and hence a map  $\Theta:\Aff^+(x,\omega) \rightarrow \SL(2,\Z[\alpha])$.
Moreover, this basis descends to an $\F_p(\overline{\alpha})$-basis $\{\overline{\gamma},\overline{\delta}\}$ of $H_1(X;\F_p)$; we get an identification $\Aut^+_{\F_p(\overline{\alpha})}(H_1(X;\F_p)) \cong \SL(2,\F_p(\overline{\alpha}))$ and a map $\overline{\Theta}:\Aff^+(X,\omega) \rightarrow \SL(2,\F_p(\overline{\alpha}))$.
Overall we have the following commutative diagram:
\begin{equation*}
\begin{tikzcd}
	{\Aut^+_{\Z[\alpha]}(H_1(X;\Z))} && {\Aut^+_{\F_p(\overline{\alpha})}(H_1(X;\F_p))} \\
	& {\Aff^+(X,\omega)} \\
	{\SL(2,\Z[\alpha])} && {\SL(2,\F_p(\overline{\alpha}))}
	\arrow[from=2-2, to=1-1]
	\arrow[from=2-2, to=1-3]
	\arrow["\Theta"', from=2-2, to=3-1]
	\arrow["{\overline{\Theta}}", from=2-2, to=3-3]
	\arrow[from=1-1, to=3-1]
	\arrow["\cong", from=3-1, to=1-1]
	\arrow[from=1-1, to=1-3]
	\arrow[from=3-1, to=3-3]
	\arrow[from=3-3, to=1-3]
	\arrow["\cong", from=1-3, to=3-3]
\end{tikzcd}
\end{equation*}

We claim that the elements $\overline{\Theta}(\phi_H)$ and $\overline{\Theta}(\phi_V)$ generate $\SL(2,\F_p(\overline{\alpha}))$, giving the desired surjectivity.
To begin, observe that for any $\phi \in \Aff^+(X,\omega)$ we have
\begin{equation*}
\hol(\phi_*\gamma) = D\phi \begin{pmatrix} x \\ 0 \end{pmatrix} \text{ and }
\hol(\phi_*\delta) = D\phi \begin{pmatrix} 0 \\ y \end{pmatrix}.
\end{equation*}
It follows that the map $\Theta$ is  
\begin{equation*}
\Theta(\phi) = \begin{pmatrix}
x & 0 \\ 0 & y
\end{pmatrix}^{-1}
D\phi
\begin{pmatrix}
x & 0 \\ 0 & y
\end{pmatrix}.
\end{equation*}
In particular, we have that
\begin{equation*}
\Theta(\phi_H) = 
\begin{pmatrix}
1 & c' \\ 0 & 1
\end{pmatrix}
\text{ and }
\Theta(\phi_V) =
\begin{pmatrix}
1 & 0 \\ d' & 1
\end{pmatrix}
\end{equation*}
for some $c',d' \in \Z[\alpha]$.

We now reduce mod $p$.
Letting $\overline{\cdot}$ denote the projection $\Z[\alpha] \rightarrow \F_p(\overline{\alpha})$, we have that
\begin{equation*}
\overline{\Theta}(\phi_H) =
\begin{pmatrix}
1 & \overline{c'} \\ 0 & 1
\end{pmatrix}
\text{ and }
\overline{\Theta}(\phi_V) =
\begin{pmatrix}
1 & 0 \\ \overline{d'} & 1
\end{pmatrix}.
\end{equation*}
Since
\begin{equation*}
c'd' = \tr(\Theta(\phi_H)\Theta(\phi_V))-2 = \tr(D\phi_H D\phi_V) - 2= cd = \alpha,
\end{equation*}
after conjugating by $\left( \begin{smallmatrix} \overline{d'} & 0 \\ 0 & 1 \end{smallmatrix} \right)$ we obtain the matrices
\begin{equation*}
\begin{pmatrix}
1 & \overline{c'd'} \\ 0 & 1
\end{pmatrix}
=
\begin{pmatrix}
1 & \overline{\alpha} \\ 0 & 1
\end{pmatrix}
\text{ and }
\begin{pmatrix}
1 & 0 \\ 1 & 1
\end{pmatrix}.
\end{equation*}
We now appeal to a classical theorem of Dickson (see Gorenstein \cite[Chap. 2, Theorem 8.4] {Gorenstein}): since $p^g \neq 9$, these two matrices generate $\SL(2,\F_p(\overline{\alpha}))$.
Hence $\overline{\Theta}(\phi_H)$ and $\overline{\Theta}(\phi_V)$ generate $\SL(2,\F_p(\overline{\alpha}))$ as desired.
\end{proof}
\begin{remark}\label{rem:Dickson_exception}
    The above argument does not address the case when $\F_{p^g} \cong \F_9$, i.e., when $g = 2$ and $p = 3$.
    Dickson's theorem also addresses this situation: the subgroup
    \begin{equation*}
        \left\langle
        \begin{pmatrix}
        1 & \overline{\alpha} \\ 0 & 1
        \end{pmatrix},
        \begin{pmatrix}
        1 & 0 \\ 1 & 1
        \end{pmatrix}
        \right\rangle
    \end{equation*}
    is isomorphic to $\SL(2,5)$.
    We will use this fact to compute the degree of the congruence cover $B_3 \to C$ in this exceptional situation, see \autoref{rem:double_pentagon}.
\end{remark}

\begin{remark}
    In the case that $m_\alpha(x)$ is not irreducible over $\F_p$, the proof of \autoref{thm:action_on_homology_mod_p} still shows that $\rho_m$ factors through $\SL(2,R)$, where $R = \F_p[x]/(m_\alpha(x))$.
    However, in this case, we cannot apply Dickson's theorem to conclude that $\Aff^+(X,\omega)$ surjects onto $\SL(2,R)$.
\end{remark}

\subsection{The Thurston--Veech construction}\label{subsec:TV_construction}
Following \autoref{subsec:action-homology-char-p}, to compute the image of $\Aff^+(X, \omega)$ in $\Sp(H_1(X; \F_p))$ when $(X, \omega)$ is algebraically primitive, it suffices to verify that $\Im(\hol)$ is a free $\Z[\alpha]$-module of rank 2, where $\alpha$ is determined by two parabolic elements as above.
We now work towards showing that all known algebraically primitive Veech surfaces satisfy this criterion.

We will leverage the \textit{Thurston--Veech construction}: a combinatorial construction that builds a flat surface $(X,\omega)$ using curves on an abstract genus $g$ surface.
All Teichm\"uller curves arise from this process; that is, any Veech surface is $\SL(2,\R)$-equivalent to a Veech surface built from the Thurston-Veech construction.
We briefly describe the construction now, appealing to McMullen \cite[\S 6]{McMullen_survey_2022} for more details
(we present only the ``unweighted'' version of the construction, as it is sufficient for our purposes).
See \autoref{fig:BM} and \autoref{fig:sporadics} for some examples.

The input to the construction is a pair of oriented filling multicurves $(A, B)$ on a surface $X$.
This information can be recorded in a bipartite graph: the vertices correspond to the components of $A$ and $B$, which are colored black or white depending on which multicurve they belong to, and two vertices are joined by an edge if the corresponding curves intersect.
The output of the construction is a horizontally and vertically periodic translation surface structure $(X, \omega)$ on $X$ such that:
\begin{itemize}
    \item the horizontal cylinder core curves are the components of $A$,
    \item the vertical cylinder core curves are the components of $B$,
    \item all of the cylinders have the same inverse modulus $\mu$, leading to global horizontal and vertical multiwists $\phi_H, \phi_V \in \Aff^+(X, \omega)$ about $A$ and $B$ with derivatives
    \begin{equation*}
        D\phi_H = \begin{pmatrix} 1 & \mu \\ 0 & 1 \end{pmatrix}, \qquad D\phi_V = \begin{pmatrix} 1 & 0 \\ -\mu & 1 \end{pmatrix}.
    \end{equation*}
\end{itemize}

We now briefly describe how to build the translation surface $(X, \omega)$.
Let $A = \bigcup_{i = 1}^k C_i$ and $B = \bigcup_{i = k + 1}^n C_i$ denote the components of the two multicurves, and let $h_i$ be a variable denoting the height of the cylinder realizing $C_i$ as a core curve.
For each intersection point between curves $C_k$ and $C_\ell$, the surface $(X, \omega)$ has a rectangle $R_{k\ell} = [0, h_k] \times [0, h_\ell]$. 
The rectangles are connected together according the intersection pattern of the curves $\{C_k\}$.

To solve for the heights $h_k$, we set up a linear system.
Let the matrix $Q_{ij} = |C_i \cap C_j|$ record the intersection pattern of $A$ and $B$.
One can check that a solution to the system $Q \mathbf{h} = \mu \mathbf{h}$, where $\mathbf{h} = (h_1, \ldots, h_n)^t$ and $\mu \in \R$, imposes the condition that all of the cylinders have inverse modulus $\mu$.
As $Q$ is an integer matrix with nonnegative entries, the Perron--Frobenius theorem gives that $Q$ has a maximal positive eigenvalue $\mu$ with a unique (up to scale) eigenvector $\mathbf{h}$.
We use this vector $\mathbf{h}$ to determine the dimensions of the rectangles $R_{k \ell}$ so that $(X, \omega)$ supports the multitwists $\phi_H$ and $\phi_V$.

\subsection{Weierstrass eigenforms in genus 2}\label{subsec:l_tables}
McMullen \cite{McMullen_Billiards_Tcurves_Hilbert_modular_surfaces, DiscSpin} described and classified infinitely many Veech surfaces in the stratum $\Omega \mathcal{M}_2(2)$.
They generate an infinite series of Teichm\"uller curves in $\mathcal{M}_2$ called the \emph{Weierstrass curves}.
Their generating 1-forms have $L$-shaped polygon representatives that depend on two integer parameters $(w, e) \in \Z^2$ satisfying certain arithmetic conditions.
Specifically, the $L$-shaped polygon is a $\lambda \times \lambda$ square and a $w \times 1$ rectangle, where $\lambda \coloneqq (e + \sqrt{e^2 + 4w})/2$ is the positive root of $\lambda^2 = e\lambda + w$.
McMullen showed that the resulting Veech surface admits real multiplication by a quadratic order of \emph{discriminant} $D \coloneqq e^2 + 4w$.
See \autoref{fig:l-table} for the $L$-shaped polygon.
\begin{figure}[ht]
    \centering
    \includegraphics[scale=0.60]{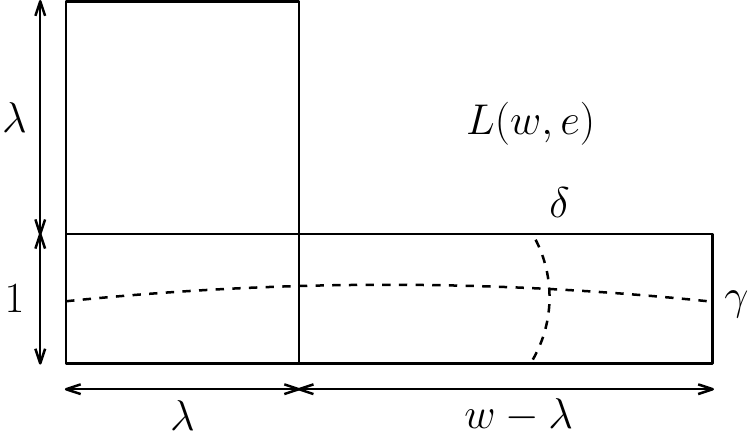}
    \caption{The L-shaped polygon generating the Weierstrass curves in genus 2.
    We use the core curves $\gamma$ and $\delta$ in verifying the hypotheses of \autoref{thm:action_on_homology_mod_p}}
    \label{fig:l-table}
\end{figure}

Given a discriminant $D$, consider the choices of parameters  $(w, e) = ((D-e^2)/4, e)$ where $e$ is $0$ or $\pm 1$ satisfying $e \equiv D \pmod{2}$.
Let $(X_D, \omega_D)$ be the resulting Veech surface and $C_D$ the Teichm\"uller curve.
In this section we assume that $D$ is nonsquare, so that $(X_D, \omega_D)$ is not square-tiled, has quadratic trace field, and is algebraically primitive.

\begin{theorem}\label{thm:invariants-l-tables}
Fix a nonsquare discriminant $D$ and choose a prime $p > 3$ for which $D$ is a quadratic nonresidue.
Let $B_{D,p} \rightarrow C_D$ be the level $p$ congruence cover and $\Xt_{D,p} \rightarrow \overline{B}_{D,p}$ the level $p$ congruence Veech fibration with fiber $(X_D,\omega_D)$.
Then:
\begin{enumerate}[label=(\roman*)]
    \item The number of cusps on $B_{D,p}$, the genus of $B_{D,p}$, the Euler characteristic of $\Xt_{D,p}$, and the signature of $\Xt_{D,p}$ are the values in \autoref{tab:l-tables}.
    \item We have $\pi_1(\Xt_{D,p}) \cong \pi_1(\overline{B}_{D,p})$.
    \item The complex surface $\Xt_{D, p}$ is minimal and of general type with strict BMY-inequality.
\end{enumerate}

\begin{table}[ht]
\begin{tabular}{||c| c ||}
\hline
$\textup{cusps}(B_{D, p})$        & $\frac{d}{p} \v \mathcal{P}_D \v$                                                \\
\hline
$\textup{genus}(B_{D, p})$        & $1 - \frac{d}{2p} \v \mathcal{P}_D \v - \frac{d}{2} \chi(C_D)$                                         \\
\hline
$T(\Xt_{D, p})$        & $d \sum_{(w, h, t, e) \in \mathcal{P}_D}(1 + \frac{h}{w})\lcm(1, w/h)$ \\
\hline
$e(\Xt_{D, p})$   & $-2\left(d\chi(C_D) + \frac{d}{p} \v \mathcal{P}_D \v\right) + T$                                            \\
\hline
$\sigma(\Xt_{D, p})$ & $-\frac{4d}{9} \chi(C_D) - \frac{2}{3} T$                                             \\ \hline
\end{tabular}
\caption{Invariants for the genus 2 Weierstrass Veech fibrations $\Xt_{D, p}$.
Here $d = | \PSL(2, p^2) | = \frac{p^2(p^4-1)}{2}$, and $\chi(C_D)$ is the Euler characteristic of the Teichm\"uller curve $C_D$.
The set $\mathcal{P}_D$ is defined below.}
\label{tab:l-tables}
\end{table}
\end{theorem}

Our formulas are in terms of the Euler characteristic of the Teichm\"uller curve $\chi(C_D)$.
Bainbridge \cite{Bainbridge2007} computed $\chi(C_D)$ in terms of the Euler characteristic of a certain Hilbert modular surface; in particular, he gives a computable formula for $\chi(C_D)$ via number-theoretic quantities such as the zeta function of $\Q(\sqrt{D})$.

Note that when $p = 3$, much of the proof of \autoref{thm:invariants-l-tables} goes through; see \autoref{rem:double_pentagon}.
Note also that in discriminant $D = 5$, the resulting $L$-table is the ``golden $L$'' that generates the double regular pentagon curve, and in discriminant $D = 8$ the resulting $L$-table generates the regular octagon curve.
We work out their invariants from another perspective in \autoref{subsec:b-m}.

\subsubsection{Degree of congruence cover}
We start by verifying that the $(X_D, \omega_D)$ satisfies the hypotheses of \autoref{thm:action_on_homology_mod_p}, so we can compute the degree of the level $p$ congruence cover.

\begin{lemma}\label{prop:l_table_action_modp_homology}
Let $p$ be a prime as in \autoref{thm:invariants-l-tables}.
Then, we have
\begin{equation*}
\Im\left(\Aff^+(X_D,\omega_D) \xrightarrow{\rho_p} \Sp(H_1(X_D;\F_p))\right) \cong \SL(2,\F_{p^2}).
\end{equation*}
\end{lemma}
\begin{proof}
We compute that the horizontal and vertical multitwists on $(X_D, \omega_D)$ have derivatives
\begin{equation*}\label{eq:l_table_multitwists}
U \coloneqq
\begin{pmatrix}
1 & w \\
0 & 1
\end{pmatrix}, \qquad
L \coloneqq
\begin{pmatrix}
1 & 0 \\ \frac{\lambda + 1}{\lambda} & 1
\end{pmatrix}.
\end{equation*}
As in \autoref{thm:action_on_homology_mod_p}, set $\alpha \coloneqq w\frac{\lambda+1}{\lambda}$.
Since $\frac{w}{\lambda} = \lambda - e$, we have
\begin{equation*}
\alpha = w \frac{\lambda + 1}{\lambda} = (\lambda - e)(\lambda + 1) = \lambda^2 - e\lambda + \lambda - e = (e\lambda + w ) - e\lambda + \lambda - e = \lambda + (w - e).
\end{equation*}

We can see that the four horizontal and vertical core curves are a generating set for $\pi_1(X_D)$ and a basis for $H_1(X; \Z)$.
It follows that
\begin{align*}
\Im(\hol) =
\left\langle
\begin{pmatrix}
    w \\
    0
\end{pmatrix},
\begin{pmatrix}
    \lambda \\
    0
\end{pmatrix},
\begin{pmatrix}
    0 \\
    1
\end{pmatrix},
\begin{pmatrix}
    0 \\
    1 + \lambda
\end{pmatrix}
\right\rangle.
\end{align*}
Assuming that $w$ is coprime to $p$ (which we check below), we conclude that the curves $\gamma$ and $\delta$ in \autoref{fig:l-table} form an $\F_p(\overline{\alpha})$-basis of $H_1(X;\F_p)$.
Following \autoref{thm:action_on_homology_mod_p}, when $m_\alpha(x)$ is irreducible in $\F_p[x]$ we have that the image of $\Aff^+(X_D, \omega_D)$ in $\Sp(H_1(X_D; \F_p))$ has order $|\SL(2, \F_{p^2})|$.

It remains to classify for which primes $p$ the minimal polynomial $m_\alpha(x)$ is irreducible over $\F_p$.
We compute that
\begin{align*}
    m_\alpha(x) = x^2 + (e - 2w)x + w(w - e - 1) = \left(x - \frac{e - 2w}{2}\right)^2 - \frac{e^2 + 4w}{4}.
\end{align*}
Since $e^2 + 4w = D$, we see that $m_\alpha(x)$ is irreducible over $\F_p$ exactly when $D/4$, and hence $D$, is a quadratic nonresidue in $\F_p$.
Note that if $w$ were a multiple of $p$, then $D / 4 \equiv e^2 / 4$ would be a quadratic residue.
\end{proof}

It follows that the level $p$ congruence cover $B_{D,p} \rightarrow \widehat{C}_D$ has degree $d = |\PSL(2, \F_{p^2})|$.

\subsubsection{Cusps and genus}
Following McMullen \cite[Theorem 4.1]{DiscSpin}, as $D$ is nonsquare we have that the cusps of the Teichm\"uller curve $C_D$ are in bijection with the set of all \emph{integer prototypes}
\begin{align*}
    \mathcal{P}_D \coloneqq \{(w, h, t, e) \in \Z^4:\ & D = e^2 + 4wh, w > 0, h > 0,\\
    & 0 \le t < \gcd(w, h), h + e < w, \gcd(w, h, t, e) = 1\}.
\end{align*}

\begin{remark}
    When $D \equiv 1 \pmod{8}$, we must consider a $\Z/2\Z$-valued invariant of a prototype called \emph{spin} when counting the elements of $\mathcal{P}_D$.
    In the 1 mod 8 case, we only include prototypes with the same spin as the standard parameter choice $(w, 1, 0, e)$ as above.
    See McMullen \cite[\S 5]{DiscSpin} for more information about spin and its relationship to the Weierstrass curves.
    For our purposes, we elide the issue by letting $\mathcal{P}_D$ denote prototypes with the correct spin.
\end{remark}

In particular, the number of cusps on $C_D$ is $|\mathcal{P}_D|$.
The prototype $(w, h, t, e)$ corresponds to a cylinder decomposition of $(X_D, \omega_D)$ into two cylinders: a short cylinder of inverse modulus 1, and a long cylinder of inverse modulus $w / h$.
If we write $c \coloneqq \lcm(1, w/h)$, the twisting about this cusp is $c + c(h/w)$.
One checks that the multitwist is the cusp generator (as $D$ is nonsquare, there are no fractional multitwists), and that the multitwist acts nontrivially mod $p$ (following the formula \cite[Prop~6.3]{farb_margalit}). 
It follows that each cusp generator of $C_D$ maps to an order $p$ element of $\Sp(H_1(X;\F_p))$. 
Letting $\Delta_{D,p}$ denote the set of cusps on $B_{D,p}$, we have that
\begin{equation*}
    \v \Delta_{D,p} \v = \frac{d}{p} \v \mathcal{P}_D \v.
\end{equation*}
The total twisting is
\begin{equation*}
    T(\Xt_{D, p}) = d \sum_{(w, h, t, e) \in \mathcal{P}_D}\left(1 + \frac{h}{w}\right)\lcm(1, w/h).
\end{equation*}

To compute the genus $b$ of $B_{D,p}$, we appeal to the fact that (orbifold) Euler characteristic is multiplicative.
That is, we have that $\chi(B_{D,p}) = d\chi(C_D)$; in particular, it follows that
\begin{equation*}
    2 - 2b - \frac{d}{p}\v \mathcal{P}_D \v = d\chi(C_D).
\end{equation*}

\subsubsection{Topological invariants}
Next we compute $e(\Xt_{D,p})$ and $\sigma(\Xt_{D,p})$ following \autoref{prop:euler_char} and \autoref{prop:signature_formula}.
Since $(X_D, \omega_D) \in \Omega \mathcal{M}_2(2)$, we have that $\kappa_\mu = 2/9$.
We can now plug these quantities into \autoref{prop:euler_char} and \autoref{prop:signature_formula} to compute that
\begin{equation*}
    e(\Xt_p) = -2(d\chi(C_D) + \v \Delta_{D,p} \v) + T
\end{equation*}
and
\begin{equation*}
    \sigma(\Xt_p) = -\frac{4}{9} \cdot d \chi(C_D) - \frac{2}{3} T.
\end{equation*}

The fact that $\pi_1(\X_{D,p}) \cong \pi_1(\overline{B}_{D,p})$ follows directly from \autoref{prop:pi_1_suff_cond}.

\subsubsection{Geometric invariants}
    Note that the BMY-inequality is strict by \autoref{prop:alg_prim_BMY}.
    To show that $\Xt_{D, p}$ is general type, it suffices by \autoref{thm:kodaira-dimension-genus-high} to show that the genus $b$ of $B_{D, p}$ is at least 1.
    When $D > 41$, $C_D $ has genus greater than 0 (see \cite[\S 4]{McMullen_survey_2022}, and for all smaller discriminants $C_D$ has genus 0.
    
    When $D = 5$, we have $\chi(C_D) = -3/10$ and $C_D$ has one cusp.
    So
    \begin{align*}
        b = 1 - \frac{d}{2p}(1) - \frac{d}{2}(-3/10) = d\left(\frac{3}{20} - \frac{1}{2p}\right) + 1.
    \end{align*}
    We see that $b \ge 1$ when $p \ge 5$.
    We remark that $b = 0$ when $p = 3$.

    When $D = 8$, we have $\chi(C_D) = -3/4$ and $C_D$ has two cusps.
    We have
    \begin{align*}
        b = 1 - \frac{d}{2p}(2) - \frac{d}{2}(-3/4) = d\left(\frac{3}{8} - \frac{1}{p}\right) + 1.
    \end{align*}
    It follows that for all $p \ge 3$ we have $b \ge 1$.
    
    Now assume $8 < D \le 41$.
    Per our formula in \autoref{tab:l-tables}, we must show that
    \begin{align*}
        &1 - \frac{d}{2p} \v \mathcal{P}_D \v - \frac{d}{2} \left(2 - 2(0) - \v \mathcal{P}_D \v - e_2(C_D)\left(1 - \frac{1}{2}\right) \right) \ge 1,
    \end{align*}
    where $e_2(C_D)$ is the number of order 2 orbifold points on $C_D$.
    (Mukamel \cite{Mukamel_2014} gave an expression for $e_2(C_D)$ in terms of class numbers of number fields.)
    An equivalent inequality is $\v \mathcal{P}_D\v / 2 + e_2(C_D)/4 - \v \mathcal{P}_D\v / 2p \ge 1$.
    Computer calculation verifies the inequality for $p \ge 3$.

Using known values on the Teichm\"uller curves $C_D$ (see, e.g., McMullen \cite[Appendix C]{McMullen_survey_2022}), we computed the Chern numbers of $\Xt_p$ for discriminants up to 60 and various small primes; see \autoref{fig:l-table-chern-nos}.
In each figure, the top line is the BMY line and the lower line is the Noether line.
It appears that the ratio of Chern numbers grows approximately linearly in $D$; it would be interesting to theoretically verify whether this is the case.
\begin{figure}
    \begin{subfigure}{\textwidth}
        \centering
        \hspace*{-1cm}
        \includegraphics[scale=0.60]{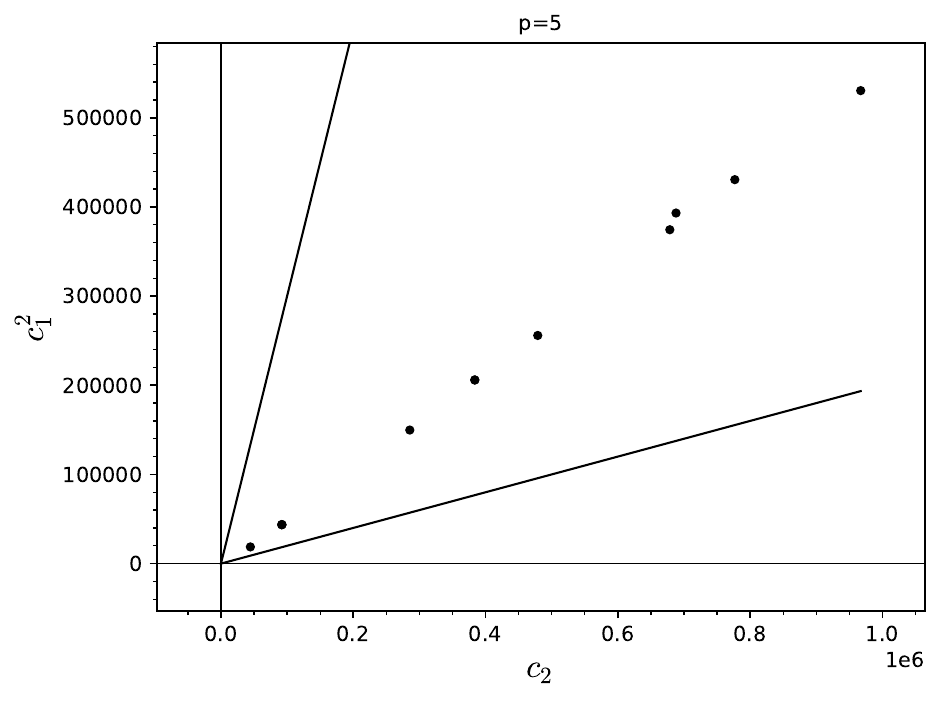}
    \end{subfigure}
    \begin{subfigure}{\textwidth}
        \centering
        \includegraphics[scale=0.56]{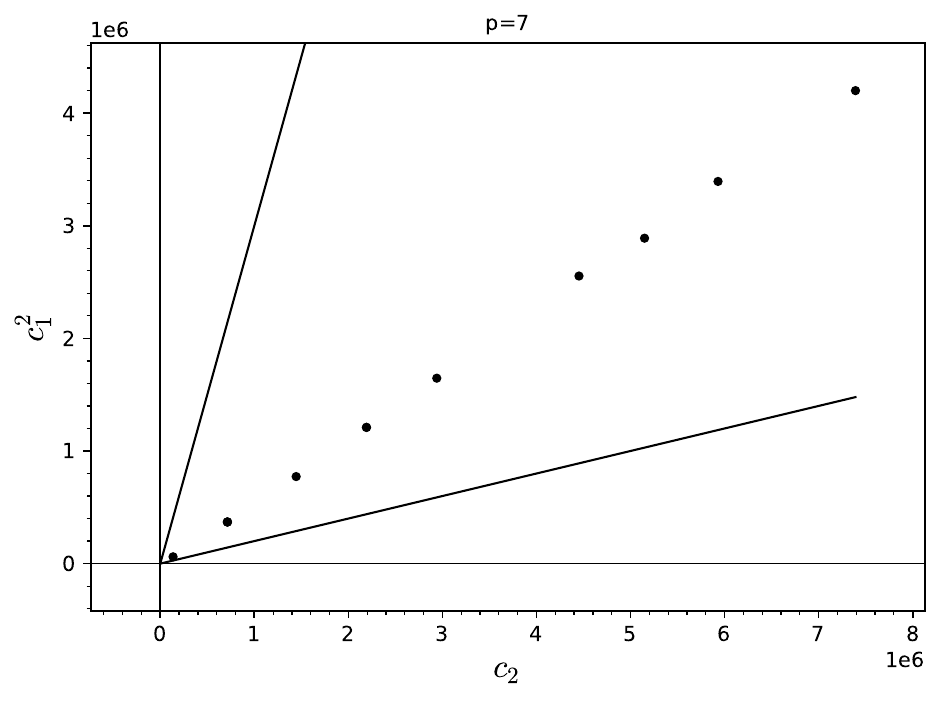}
    \end{subfigure}
    \begin{subfigure}{\textwidth}
        \centering
        \includegraphics[scale=0.56]{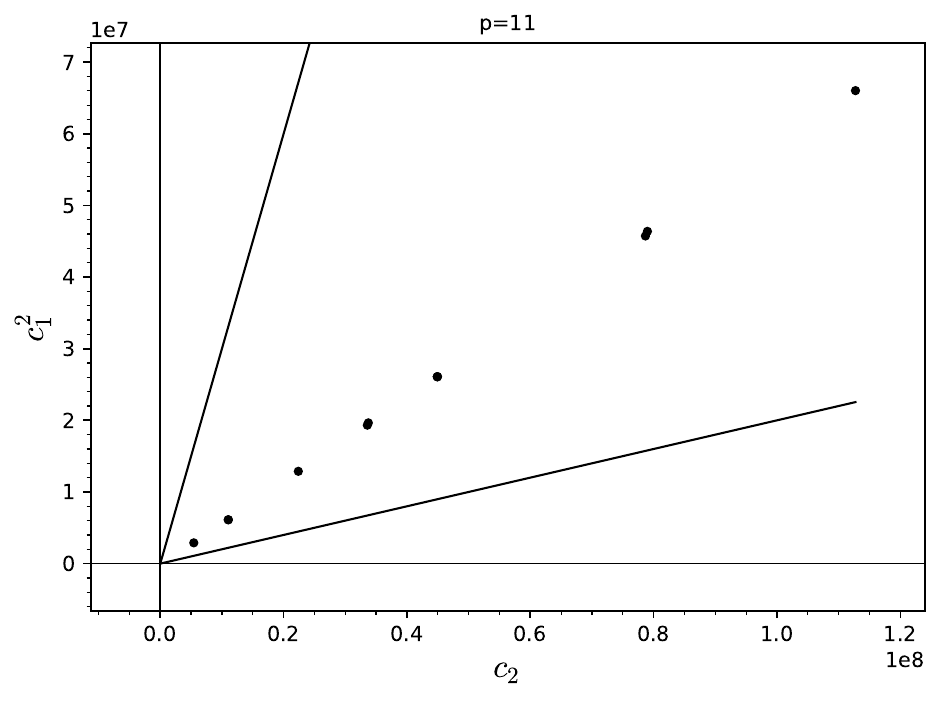}
    \end{subfigure}
    \caption{Chern numbers $(c_2, c_1^2)$ for Veech fibrations $\X_{D, p}$ with $p = 5, 7, 11$.
    We only plot discriminants $D$ for which $D$ is a nonresidue.
    Note that the scales of the axes use scientific notation; the relevant powers of 10 are at the upper left and lower right corners.} 
    \label{fig:l-table-chern-nos}
\end{figure}

\subsection{Regular polygons}\label{subsec:b-m}

For each integer $n \geq 3$, let $(Y_n,\omega_n)$ denote the regular $n$-gon surface if $n$ is even, and the double regular $n$-gon if $n$ is odd.
These surfaces, originally studied by Veech \cite{Veech}, are all geometrically primitive.
Later, Bouw and M\"oller \cite{Bouw-Moller} constructed a larger family of geometrically primitive Veech surfaces indexed by pairs of integers $2 \leq n' \leq n$; in the case $n' = 2$, they recover the surfaces $(Y_n,\omega_n)$.
Wright \cite{WrightTriangles} proved that the only pairs $(n',n)$ giving an algebraically primitive Bouw-M\"oller surface are
\begin{itemize}
    \item $(2,q)$ for $q > 3$ prime,
    \item $(2, 2q)$ for $q > 3$ a prime,
    \item $(2, 2^k)$ for $k > 2$.
\end{itemize}
We restrict our attention to $(Y_n,\omega_n)$ for these three series of $n$.

Let $C_n$ denote the Teichm\"uller curve of $(Y_n,\omega_n)$, and let $B_{n,p} \rightarrow C_n$ be the level $p$ congruence cover.
In this section, we apply the results above to compute the invariants of the congruence Veech fibration $\Xt_{n,p} \rightarrow \overline{B}_{n,p}$ for certain primes $p$.
We obtain the following.

\begin{theorem}\label{thm:BM_computations}
    Let $n = q$, $n = 2q$, or $n=2^k$ for $q > 3$ prime and $k > 2$.
    Let $\alpha = 4\cos\left(\frac{\pi}{n}\right)^2$, and let $p \geq 3$ be a prime such that the minimal polynomial of $\alpha$ is irreducible over $\F_p$, and $(p,g) \neq (3,2)$ where $g$ is the genus of $Y_n$.
    Let $\Xt_{n,p} \rightarrow \overline{B}_{n,p}$ be the level $p$ congruence Veech fibration of $(Y_n,\omega_n)$.
    Then, 
    \begin{enumerate}[label=(\roman*)]
        \item the genus of $B_{n,p}$, the number of cusps of $B_{n,p}$, the Euler characteristic of $\Xt_{n,p}$, and the signature of $\Xt_{n,p}$ are given by Tables \ref{tab:BM_q_and_2q} and \ref{tab:BM_2tok},
        \item $\pi_1(\Xt_{n,p}) \cong \pi_1(\overline{B}_{n,p})$,
        \item $\Xt_{n,p}$ is a minimal general type surface with strict BMY-inequality.
    \end{enumerate}
\end{theorem}
\begin{table}[ht]
\begin{tabular}{||c| c c ||} 
 \hline
  & $n=q$ & $n=2q$  \\  
 \hline
 genus($B_{n,p}$) & $1 + \frac{d}{2}\left(\frac{1}{2} - \frac{1}{q} - \frac{1}{p}\right)$ & $1 + d\left(\frac{1}{2} - \frac{1}{2q} - \frac{1}{p}\right)$  \\ 
 \hline
 cusps($B_{n,p}$) & $\frac{d}{p}$ & $\frac{2d}{p}$  \\
 \hline
 $e(\Xt_{n,p})$ & $d\left((q-3)\left(\frac{1}{2}-\frac{1}{q}-\frac{1}{p}\right)+\frac{q-1}{2}\right)$ & $d\left((2q-6)\left(\frac{1}{2}-\frac{1}{2q}-\frac{1}{p}\right) + q\right)$  \\
 \hline
 $\sigma(\Xt_{n,p})$ & $-\frac{d(q^2-1)}{4q}$ & $-\frac{d(q^2+2q+3)}{3q}$ \\
 \hline
\end{tabular}
\caption{The invariants of $B_{n,p}$ and $\Xt_{n,p}$ for $n = q$ or $n=2q$ where $q > 3$ is prime.  Here $d = \v\PSL(2,\F_{p^g})\v = \frac{p^g(p^{2g}-1)}{2}$ is the degee of the congruence cover $B_{n,p} \rightarrow \widehat{C}_n$.}
\label{tab:BM_q_and_2q}
\end{table}

\begin{table}[ht]
    \centering
    \begin{tabular}{||c|c||}
    \hline
     & $n=2^k$ \\
     \hline
     genus($B_{n,p}$) & $1 + d\left(\frac{1}{2} - \frac{1}{2^k} - \frac{1}{p}\right)$ \\
     \hline
     cusps($B_{n,p}$) & $\frac{2d}{p}$ \\
     \hline
     $e(\Xt_{n,p})$ & $d\left((2^k-4)\left(\frac{1}{2}-\frac{1}{2^k}-\frac{1}{p}\right) + 2^{k-1}\right)$ \\
     \hline
     $\sigma(\Xt_{n,p})$ & $-d(2^{k - 2} + \frac{1}{3})$ \\
     \hline
    \end{tabular}
    \caption{The invariants of $B_{n,p}$ and $\Xt_{n,p}$ for $n = 2^k$ where $k > 2$.  Here $d = \v\PSL(2,\F_{p^g})\v = \frac{p^g(p^{2g}-1)}{2}$ is the degree of the congruence cover $B_{n,p} \rightarrow \widehat{C}_n$.}
    \label{tab:BM_2tok}
\end{table}

\begin{remark}\label{rem:double_pentagon}
    For the exceptional case $(p,g) = (3,2)$, the proofs of parts (i) and (ii) in \autoref{thm:BM_computations} go through almost identically.
    The only issue is that in this case, the degree $d$ of the congruence cover $B_{n,3} \rightarrow \widehat{C}_n$ is not $\v \PSL(2,\F_{9}) \v$.
    Rather, since the parabolics $\left( \begin{smallmatrix} 1 & 2\cos(\pi/n) \\ 0 & 1 \end{smallmatrix} \right)$ and $\left( \begin{smallmatrix} 1 & 0 \\ 2\cos(\pi/n) & 1 \end{smallmatrix} \right)$ generate $\SL(X,\omega)$ (see below), the discussion in \autoref{rem:Dickson_exception} tells us that $d = \v \PSL(2,5) \v = 60$ in this case.

    The case $(p,g) = (3,2)$ is realized for the double pentagon surface $(Y_5,\omega_5)$ and its level 3 congruence Veech fibration $\Xt_{5,3} \rightarrow B_{5,3}$.
    Note that the $(Y_5, \omega_5)$ is in the same $\SL(2, \R)$-orbit as the \textit{golden L}, i.e., the genus 2 Weierstrass L-shaped eigenform of discriminant $5$.
    In this case, $B_{5,3}$ has genus 0; the cover $B_{5,3} \rightarrow \widehat{C}_5$ is topologically equivalent to the quotient of the regular dodecahedron by its symmetry group $A_5$.
    The total space $\Xt_{5,3}$ is simply connected, and we can compute that $e(\Xt_{5,3}) = 116$ and $\sigma(\Xt_{5,3}) = -72$.
    While our minimality argument for high-genus fibrations in \autoref{thm:kodaira-dimension-genus-high} does not apply to this example, Dawei Chen has suggested to us the following direct argument that $\Xt_{5,3}$ is minimal.

    \begin{lemma}
        The level $3$ congruence Veech fibration $\Xt_{5,3} \rightarrow \overline{B}_{5,3}$ for the double pentagon surface $(Y_5, \omega_5)$ is minimal.
    \end{lemma}
    \begin{proof}
    Suppose for contradiction that there is a non-singular rational curve $S \rightarrow \Xt_{5, 3}$ with self-intersection $-1$.
    Since $\Xt_{5, 3} \rightarrow \overline{B}_{5, 3}$ is relatively minimal, we know that $S$ is not contained in a single fiber.
    It follows that the restriction $p\v_S:S \rightarrow \overline{B}_{5, 3}$ is non-constant, and hence a branched cover.
    So $S$ is in fact a multisection of $\Xt_{5, 3} \rightarrow \overline{B}_{5, 3}$.
    
    Note that for any multisection of $\Xt_{5, 3} \rightarrow \overline{B}_{5, 3}$, the set of points of the multisection in any smooth fiber must be closed under the action of $\Aff^+(Y_5,\omega_5)[3]$.
    Using M\"oller's classification of periodic points \cite{moller_periodic_2006}, we can show that $\Aff^+(Y_5,\omega_5)$ has two periodic orbits on $(Y_5,\omega_5)$: the singleton set consisting of the zero of $\omega_5$, and the set of the remaining 5 Weierstrass points $p_1, \ldots, p_5$.
    One can check that $\Aff^+(Y_5,\omega_5)[3]$ acts transitively on the 5 non-zero Weierstrass points $p_i$, and hence $\Aff^+(Y_5, \omega_5)[3]$ has the same periodic orbits.

    Since $S$ is connected, there are two possibilities: either $S$ is the zero section, or $S$ is the multisection given by the points $p_i$.
    Since we can compute that the zero section has self-intersection $-3$ by \autoref{prop:zero_section}, it must be the latter.
    In particular, $p\v_S$ has degree 5, is non-ramified at any point in a smooth fiber.
    Note that since $p\v_S$ is ramified, this means it must be ramified at some point in some singular fiber.
    In fact, since each cusp of $B_{5, 3}$ is a preimage of the lone cusp on the original Teichm\"uller curve $C_5$, $p\v_S$ is ramified at the same number of points $k \ge 1$ in each singular fiber of $\Xt_{5, 3}$.
    
    We now get a contradiction as follows.
    Riemann-Hurwitz says that
    \begin{equation*}
        2 - r = 5(2 - b),
    \end{equation*}
    where $b$ is the number of branch points on $\overline{B}_{5, 3}$, and $r$ is the total number of points on $S$ at which $p\v_S$ is ramified.
    From above, we know that $p\v_S$ is ramified at a point in each singular fiber, and hence $b = 20$ since $B_{5, 3}$ has 20 cusps.
    Moreover, the ramification is $r = 20k$.
    But the equation
    \begin{equation*}
        2 - 20k = 5(2 - 20)
    \end{equation*}
    has no positive integer solutions.
    \end{proof}
    
    We conclude that $\Xt_{5,3}$ is a \emph{Horikawa surface}: a minimal general type complex surface on the Noether line $c_1^2 = 2p_g - 4$.
    Following \cite{Auroux}, up to deformation equivalence there are precisely two minimal general type surfaces with $e = 116$ and $\sigma = -72$.
    These two surfaces are homeomorphic by M. Freedman's classification, but it is an open problem to determine if they are diffeomorphic or symplectomorphic.
    Both surfaces admit a second fibering with genus 17 fibers; it would be of interest to realize this second fibering on $\Xt_{5,3}$ from the perspective of translation surfaces.
    Topologically, both surfaces are double branched covers of $S^2 \times S^2$; on $\Xt_{5,3}$, this should correspond to the fiberwise quotient by the hyperelliptic involution, yielding a $\P^1$-bundle over $\P^1$ corresponding to a map $\P^1 \rightarrow \M_{0,6}$.
\end{remark}

\subsubsection{Constructing the family}
Rather than working with the polygonal surfaces $(Y_n,\omega_n)$, we follow Hooper \cite{Hooper} (see also \cite{Leininger_2004} and \cite{McMullen_survey_2022}) to build a Thurston-Veech surface $(X_n,\omega_n)$ that lies on the same Teichm\"uller curve.
Let $g'$ be $\frac{n}{2}$ if $n$ is even and $\frac{n-1}{2}$ if $n$ is odd.
On a surface of genus $g'$, choose a pair of multicurves whose intersection pattern corresponds to consider the Coxeter diagram of type $A_{n-1}$ with its natural bipartite decomposition (as in \cite[Fig~6.1]{McMullen_survey_2022}).
Let $(X_n',\omega_n')$ denote the surface obtained from applying the Thurston-Veech construction to this pair of multicurves.
If $n$ is even, the surface $(X_n',\omega_n')$ has an order 2 translation automorphism $\iota$ induced by the $180^\circ$ rotation of $A_{n - 1}$ centered on the middle vertex.
We define $(X_n,\omega_n)$ to be the quotient $(X_n',\omega_n')$ by $\iota$ if $n$ is even, and we define it to be $(X_n',\omega_n')$ if $n$ is odd.
The surface $(X_n,\omega_n)$ is staircase-shaped; see \autoref{fig:BM}.
\begin{figure}[ht]
    \centering
    \includegraphics[scale=.29]{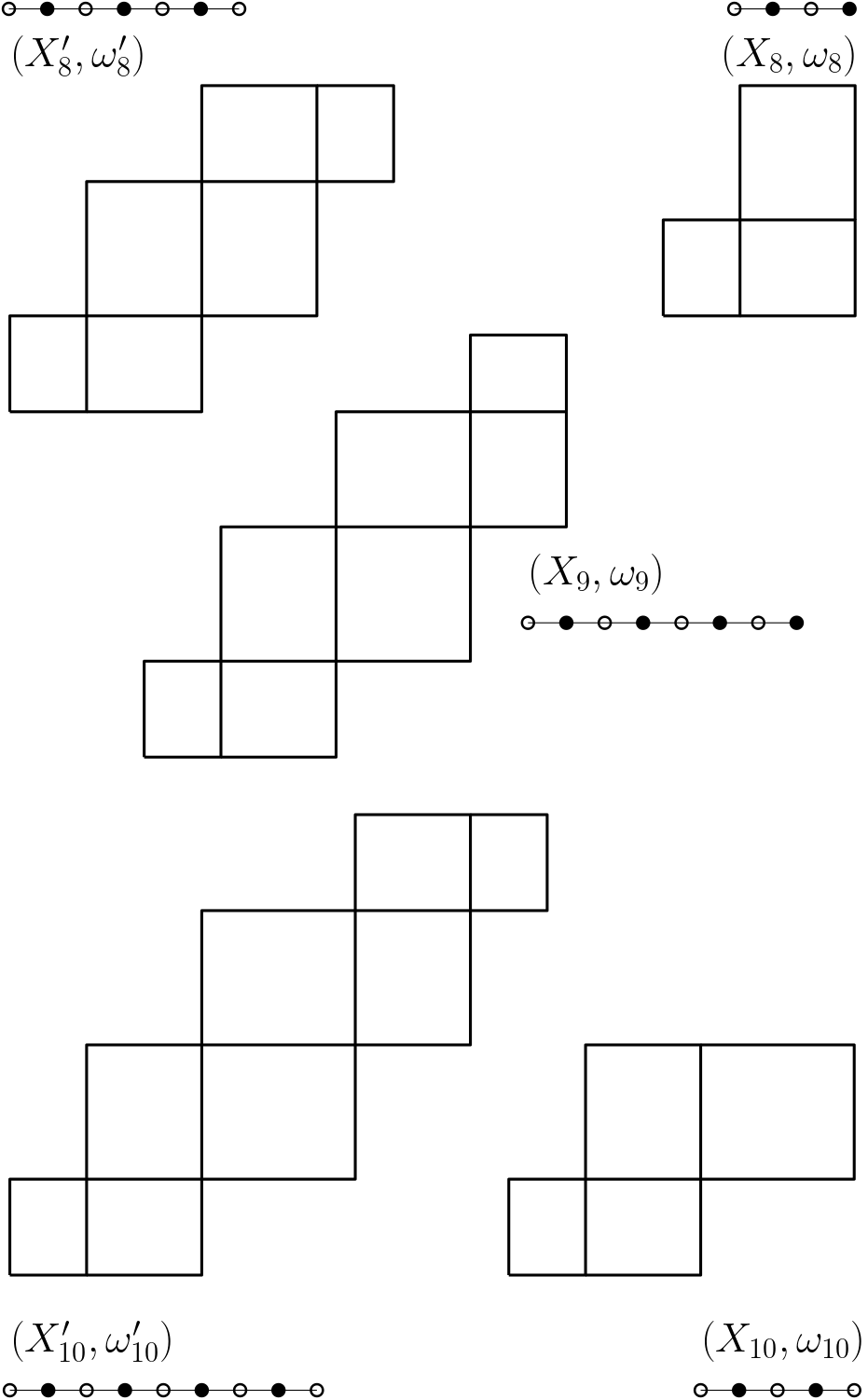}
    \caption{Examples of the surface $(X_n,\omega_n)$}
    \label{fig:BM}
\end{figure}
It has the following cylinder decompositions:
\begin{itemize}
    \item if $n$ is odd, it will have $\frac{n-1}{2}$ horizontal and vertical cylinders,
    \item  if $n = 4k$, it will have $\frac{n}{4}$ horizontal cylinders and vertical cylinders,
    \item  if $n = 4k+2$, it will have $\frac{n-2}{4}$ horizontal and $\frac{n-2}{4}+1$ vertical cylinders.
\end{itemize}
Recall that we are only interested in the case $n=q$, $n=2q$, or $n=2^k$ for $q > 3$ prime and $k > 2$.  In these cases, the stratum of $(X_n,\omega_n)$ is as follows \cite{Hooper}:
\begin{itemize}
    \item $(X_q,\omega_q)$ has genus $g=\frac{q-1}{2}$ and a single zero,
    \item $(X_{2q},\omega_{2q})$ has genus $g = \frac{q-1}{2}$ and 2 zeros of equal order,
    \item $(X_{2^k},\omega_{2^k})$ has genus $g = 2^{k-2}$ and a single zero.
\end{itemize}

By \cite{Hooper}, the surface $(X_n,\omega_n)$ lies on the same Teichm\"uller curve as $(Y_n,\omega_n)$.
The Perron-Frobenius eigenvalue is $\mu = 2\cos(\frac{\pi}{n})$.
The Veech group of $(X_n,\omega_n)$ is a triangle group; letting $\Delta(a,b,c) \subseteq \SL(2,\R)$ denote the orientation-preserving triangle group generated by elements of orders $a$, $b$, and $c$, we have that
\begin{equation*}
    \SL(X_n,\omega_n) \cong
    \begin{cases}
        \Delta(2,n,\infty) & \text{$n$ is odd} \\
        \Delta(\frac{n}{2}, \infty, \infty) & \text{$n$ is even}.
    \end{cases}
\end{equation*}
In particular, the Teichm\"uller curve $C_n$ corresponds to the hyperbolic orbifold of signature $(2,n,\infty)$ or $(\frac{n}{2},\infty,\infty)$.
In each case, we have that $-I \in \SL(X,\omega)$, and so each cusp is regular.
Each cusp corresponds to a single Dehn twist around each core curve in the corresponding cylinder decomposition.

\subsubsection{The congruence cover}
We start by showing that $(X_n, \omega_n)$ satisfies the hypotheses of \autoref{thm:action_on_homology_mod_p}, where $\alpha = \mu^2 = 4\cos(\frac{\pi}{n})^2$.

\begin{lemma}\label{lem:staircase-heights}
    Up to normalization, the horizontal cylinder heights of $(X_n, \omega_n)$ are odd integer polynomials in $\mu$, and the vertical cylinder heights of $(X_n, \omega_n)$ are even integer polynomials in $\mu$.
\end{lemma}
\begin{proof}
    First, observe that it's enough to show that all horizontal heights are odd integer polynomials in $\mu$.
    Indeed, under this assumption we get that all vertical circumferences are odd integer polynomials in $\mu$, since each vertical circumference is a sum of horizontal heights.
    Since each cylinder has inverse modulus $\mu$, each vertical circumference $c_V$ and height $h_V$ satisfies $c_V / h_V = \mu$.
    If $c_V$ is an odd integer polynomial in $\mu$, we conclude that $h_V$ is an even integer polynomial in $\mu$.
    
    We show that all horizontal heights are odd integer polynomials in $\mu$ by induction.
    We address the base case by normalizing the height of the lowest horizontal cylinder to be $\mu$.
    For our inductive step, suppose that the first $n$ horizontal cylinder heights are odd integer polynomials in $\mu$.
    Let $H$ be the $n$th horizontal cylinder with height $h_H$, and let $H'$ denote the $(n+1)$st horizontal cylinder with height $h_{H'}$.
    See \autoref{fig:staircase}.
    \begin{figure}[ht]
        \centering
        \includegraphics[scale=0.5]{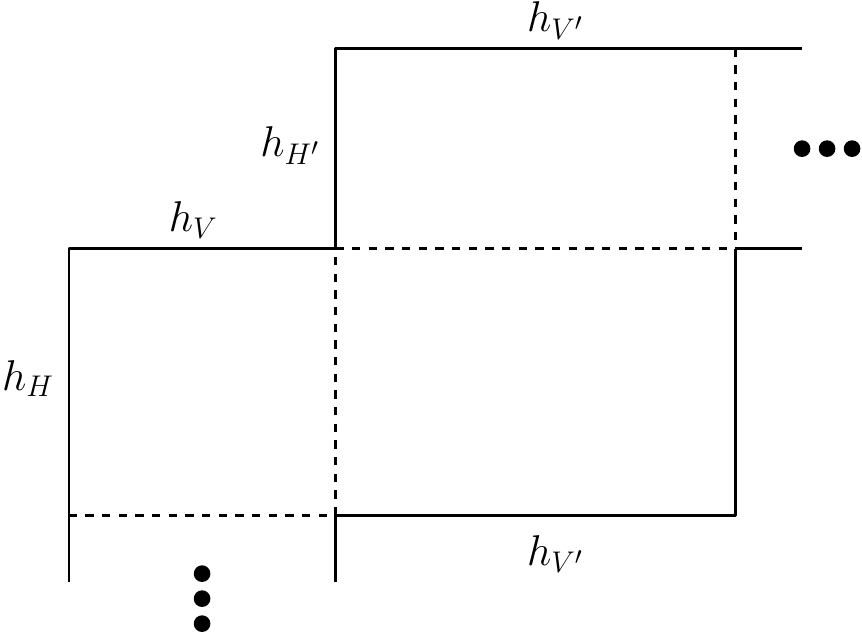}
        \caption{The inductive step of \autoref{lem:staircase-heights}}
        \label{fig:staircase}
    \end{figure}

    Our goal is show that $h_{H'}$ is also an odd integer polynomial in $\mu$.
    Let the two vertical cylinders crossing $H$ be $V$ and $V'$ with heights $h_V$ and $h_{V'}$ respectively.
    We claim that both $h_V$ and $h_{V'}$ are even integer polynomials in $\mu$.
    Since the circumference of $V$ is an odd integer polynomial in $\mu$ by induction, and $c_V / h_V = \mu$, we have that $h_V$ is an even integer polynomial in $\mu$. 
    We also have that
    \begin{align*}
        \frac{c_H}{h_H} = \frac{h_V + h_{V'}}{h_H} = \mu \implies h_{V'} = \mu h_H - h_V.
    \end{align*}
    Since $\mu h_H$ is an even integer polynomial in $\mu$, and so is $h_V$, we conclude that $h_{V'}$ is an even integer polynomial in $\mu$.
    This proves the claim.

    Now consider the vertical cylinder $V'$.
    We have
    \begin{align*}
        \frac{c_{V'}}{h_{V'}} = \frac{h_H + h_{H'}}{h_{V'}} = \mu \quad \implies h_{H'} = \mu h_{V'} - h_H.
    \end{align*}
    Since $h_H$ is an odd integer polynomial in $\mu$ by induction, and so is $\mu h_{V'}$ by the above claim, it follows that $h_{H'}$ is an odd integer polynomial in $\mu$.
    This completes the proof of the inductive step.
\end{proof}
\begin{lemma}\label{cor:staircases-satisfy-criterion}
    For the surface $(X_n, \omega_n)$, $\Im(\hol)$ is a rank $2$ free $\Z[\mu^2]$-module, with a basis consisting of a horizontal and a vertical vector.
\end{lemma}
\begin{proof}
    Applying \autoref{lem:staircase-heights}, up to normalization we have that 
    \begin{enumerate}[label=(\roman*)]
        \item all of the heights of the horizontal cylinders are odd integer polynomials in $\mu$,
        \item all of the heights of the vertical cylinders are even integer polynomials in $\mu$, and
        \item there is a horizontal cylinder with height $\mu$ and a vertical cylinder with height $1$.
    \end{enumerate}

    For each cylinder, the circumference $c$ and height $h$ satisfy $c/h = \mu$. 
    So, by (iii), there is a horizontal core curve $\gamma$ and a vertical core curve $\delta$ that have holonomy vectors $\begin{pmatrix} \mu^2 \\ 0 \end{pmatrix}$ and $\begin{pmatrix} 0 \\ \mu \end{pmatrix}$.
    The classes $\gamma$ and $\delta$ are independent over $\Z[\mu^2]$, since their holonomy vectors are independent over $\R$.
    By (i), every horizontal circumference is an element of $\mu^2\Z[\mu^2]$, and hence every horizontal core curve lies in the $\Z[\mu^2]$-span of $\gamma$.
    By (ii), every vertical circumference is an element of $\mu\Z[\mu^2]$, and hence every vertical core curve lies in the $\Z[\mu^2]$-span of $\delta$.
    Since $H_1(X;\Z)$ is spanned by the horizontal and vertical core curves, we conclude that $\{\hol(\gamma), \hol(\delta)\}$ is the desired basis of $\Im(\hol)$.
\end{proof}

Thus, by \autoref{thm:action_on_homology_mod_p}, we conclude that the degree of the congruence cover $B_{n,p} \rightarrow \widehat{C}_n$ is $d = \v \PSL(2,\F_{p^g}) \v$ for a prime $p$ as in the hypothesis of \autoref{thm:BM_computations}.

\subsubsection{Cusps and genus}
Our goal now is to compute the number of cusps, genus, and twisting of the level $p$ congruence cover $B_{n,p} \rightarrow \widehat{C}_n$ for a prime $p$ as in the statement of \autoref{thm:action_on_homology_mod_p}.

We start with the case $n=q$ for $q > 3$ prime.
Recall that $\widehat{C}_n$ is the $(2,n,\infty)$ hyperbolic orbifold.
Under the map $\PSL(X,\omega) \rightarrow \PSp(H_1(X_n;\F_p))$, the orbifold points of orders 2 and $n$ map to elements of order 2 and $n$ respectively, and the cusp maps to an element of order $p$.
Since the cusp maps to an element of order $p$, there are $\frac{d}{p}$ cusps on $B_{n,p}$.
By Riemann-Hurwitz, we can compute the genus $b$ of $B_{n,p}$ to be
\begin{equation*}
    b = 1 - \frac{d}{2}\left(\frac{1}{n} + \frac{1}{p} - \frac{1}{2}\right).
\end{equation*}
As mentioned above, the cusp generator on $C_n$ corresponds to a single twist around $g = \frac{q-1}{2}$ cylinders.
Thus, the monodromy around each cusp of $B_{n,p}$ is a $p$-fold twist around $g$ cylinders.
It follows that the twisting of $B_{n,p}$ is given by
\begin{equation*}
    T = \left(\frac{d}{p}\right)(p)(g) = dg.
\end{equation*}

Next, we consider the case that $n=2q$ for $q > 3$ prime.
Recall that $\widehat{C}_n$ is the $(\frac{n}{2},\infty,\infty)$ hyperbolic orbifold.
Under the map $\PSL(X,\omega) \rightarrow \PSp(H_1(X_n;\F_p))$, the orbifold points maps to an element of order $\frac{n}{2}$, and the cusps map to elements of order $p$.
Since the cusps map to elements of order $p$, there are $\frac{2d}{p}$ cusps on $B_{n,p}$.
By Riemann-Hurwitz, we can compute the genus $b$ of $B_{n,p}$ to be
\begin{equation*}
    b = 1 - d\left(\frac{1}{n} + \frac{1}{p} - \frac{1}{2}\right).
\end{equation*}
Since $n = 2q$, we have that $n \equiv 2 \pmod{4}$.
So, one cusp corresponds to a single twist around $g$ cylinders, while the other cusp corresponds to a single twist around $g+1$ cylinders.
The total twisting is then
\begin{equation*}
    T = \left(\frac{d}{p}\right)(p)(g+(g+1)) = d(2g+1).
\end{equation*}

Finally, we consider the case that $n=2^k$ for $k > 2$.
The number of cusps and genus of $B_{n, p}$ are the same as in the case of $n = 2q$, but the twisting term is different.
Namely, since $n \equiv 0 \pmod{4}$, each cusp corresponds to a single twist around $g$ cylinders.
It follows that the total twisting is
\begin{equation*}
    T = \left(\frac{d}{p}\right)(p)(g+g) = 2dg.
\end{equation*}

\subsubsection{Topological and geometric invariants}
We are now in a position to compute the topological invariants of $\Xt_{n,p} \rightarrow B_{n,p}$ for $n \in \{q,2q,2^k\}$ and $p$ as above.
As we have computed the genus of $B_{n, p}$, the cusps of $B_{n, p}$ and the twisting about these cusps, we can use \autoref{prop:euler_char} and \autoref{prop:signature_formula} to compute $e(\Xt_{n, p})$ and $\sigma(\Xt_{n, p})$.
These values are in \autoref{tab:BM_q_and_2q} and \autoref{tab:BM_2tok}.
We observe that $\pi_1(\Xt_{n,p}) \cong \pi_1(\overline{B}_{n,p})$ by \autoref{prop:pi_1_suff_cond}.

Now we can show that $\Xt_{n,p}$ is minimal and of general type.
By \autoref{thm:kodaira-dimension-genus-high}, it suffices to show that $B_{n,p}$ has positive genus.
In principle, one could deduce this from our formula for the genus of $B_{n,p}$, but we offer the following alternative justification.

\begin{lemma}
    Let $n$ and $p$ be as in \autoref{thm:BM_computations}.
    Then, the genus of $B_{n,p}$ is positive.
\end{lemma}
\begin{proof}
    Let $G$ be the deck group of the cover of hyperbolic orbifolds $B_{n,p} \rightarrow \widehat{C}_n$.
    If $B_{n,p}$ has genus 0, then by the well-known classification of groups actions on $S^2$ (see, e.g., \cite{Thurston1997-yo}), there are only 5 possibilities for the cover $B_{n,p} \rightarrow \widehat{C}_n$:
    \begin{enumerate}[label=(\roman*)]
        \item $G$ is cyclic of order $\ell$, and $\widehat{C}_n$ has 2 branch points of orders $\ell$ and $\ell$,
        \item $G$ is dihedral of order $2\ell$, and $\widehat{C}_n$ has 3 branch ponts of orders 2, 2, and $\ell$,
        \item $G = A_4$, and $\widehat{C}_n$ has 3 branch points of orders 2, 3, and 3,
        \item $G = S_4$, and $\widehat{C}_n$ has 3 branch points of orders 2, 3, and 4,
        \item $G = A_5$, and $\widehat{C}_n$ has 3 branch points of orders 2, 3, and 5.
    \end{enumerate}
    However, as we observed above, $\widehat{C}_n$ has 3 branch points of orders 2, $n$, and $p$ if $n = q$, and orders $\frac{n}{2}$, $p$, and $p$ if $n = 2q$ or $n=2^k$.
    Thus, the only way that $B_{n,p}$ can have genus 0 is if $n=5$ and $p=3$.
    However, in this case the surface $X_n$ has genus $2$, and so this is the exceptional case that $(p,g) = (3,2)$.
\end{proof}

Finally, the fact that the BMY-inequality is strict follows from \autoref{prop:alg_prim_BMY}.
We can also observe this directly, since the BMY-inequality can be equivalently formulated as $\sigma(\Xt_{n,p}) \leq \frac{1}{3}e(\Xt_{n,p})$, and we can see in Tables \ref{tab:BM_q_and_2q} and \ref{tab:BM_2tok} that $\sigma(\Xt_{n,p}) < 0 < e(\Xt_{n,p})$.

\subsection{The sporadic examples}\label{subsec:sporadic}
Apart from the infinitely many Weierstrass curves in genus $2$ and the infinite family of regular polygons, there are two other known examples of algebraically primitive Teichm\"uller curves.
We call these two examples the $E_7$ curve and the $E_8$ curve.
The names come from the work of Leininger \cite{Leininger_2004}: he showed that spherical Coxeter diagrams, as inputs to the Thurston--Veech construction, generate Teichm\"uller curves.
The $E_7$ and $E_8$ curves then arise from the exceptional $E_7$ and $E_8$ Coxeter diagrams under this machinery.
(These Teichm\"uller curves also arise as unfoldings of $(2\pi/9, \pi/3, 4\pi/9)$ and $(\pi/5, \pi/3, 7\pi/15)$ triangles.)
See \autoref{fig:sporadics} for the Coexter diagrams and their realizations under the Thurston--Veech construction.
\begin{figure}[ht]
    \centering
    \includegraphics[scale=0.5]{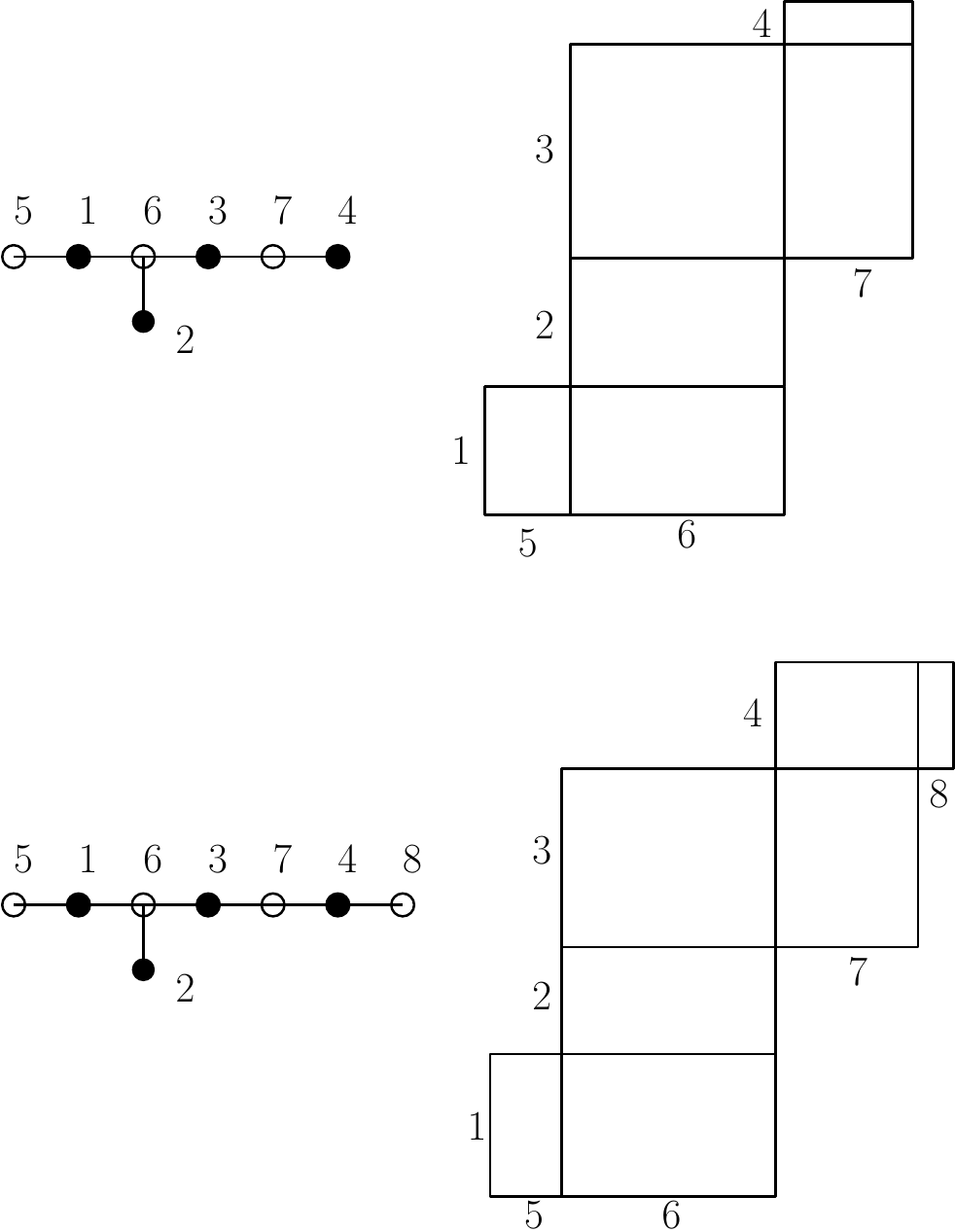}
    \caption{The $E_7$ and $E_8$ Coxeter diagrams determine translation surfaces $(X, \omega)$ that generate exceptional (algebraically primitive) Teichm\"uller curves.}
    \label{fig:sporadics}
\end{figure}
We refer to McMullen \cite[\S 6]{McMullen_survey_2022} for details on these sporadic examples.

We note that:
\begin{itemize}
    \item The $E_7$ curve has a genus 3 generating Veech surface $(X, \omega) \in \Omega \mathcal{M}_3(1, 3)$ with Veech group $\SL(X, \omega) = \Delta^-(9, \infty, \infty)$.
    The Perron--Frobenius eigenvalue $\mu_7 = 2 \cos(\pi / 18)$ has minimal polynomial $m_{\mu_7}(x) = x^6 - 6 x^4 + 9 x^2 - 3$.
    \item The $E_8$ curve has a genus 4 generating Veech surface $(X, \omega) \in \Omega \mathcal{M}_4(6)$ with Veech group $\Delta^-(15, \infty, \infty)$.
    The Perron--Frobenius eigenvalue $\mu_8 = 2 \cos(\pi /30)$ has minimal polynomial $m_{\mu_8}(x) = x^8 - 7 x^6 + 14 x^4 - 8 x^2 + 1$.
\end{itemize}
The minus superscript denotes that $\SL(X, \omega)$ does not contain $-I$; the group $\Delta^-(9, \infty, \infty)$ is an index 2 subgroup of the usual $\Delta(9, \infty, \infty)$.
This subtlety has a minor effect on the degree of the congruence cover.

\begin{theorem}\label{thm:invariants-sporadic}
    Let $(X,\omega)$ denote the $E_7$ or $E_8$ Veech surface.
    Let $\mu$ be the Perron-Frobenius eigenvalue, and let $p \geq 3$ be a prime such that the minimal polynomial of $\mu^2$ is irreducible over $\F_p$.
    Let $\Xt_p \rightarrow \overline{B}_p$ be the level $p$ congruence Veech fibration of $(X,\omega)$.  Then,
    \begin{enumerate}[label=(\roman*)]
        \item The genus of $B_p$, the number of cusps of $B_p$, the Euler characteristic of $\Xt_p$, and the signature of $\Xt_p$ are given by \autoref{tab:sporadics},
        \item $\pi_1(\Xt) \cong \pi_1(\overline{B}_p)$,
        \item $\Xt_p$ is a minimal general type surface with strict BMY inequality.
    \end{enumerate}
\end{theorem}

\begin{table}[th]
    \begin{tabular}{|| c | c c ||}
    \hline
    & $E_7$ & $E_8$ \\ \hline
    $\text{genus}(B_p)$  & $1 + \frac d2(\frac89 - \frac2p)$     & $1 + \frac d2(\frac{14}{15} - \frac2p)$      \\
    \hline
    $\text{cusps}(B_p)$  & $2d / p$     & $2d / p$     \\
    \hline
    $e(\Xt_p)$   & $d(\frac{95}{9} - \frac{8}{p})$     & $d(\frac{68}{5} - \frac{12}{p})$     \\
    \hline
    $\sigma(\Xt_p)$ & $-\frac{35}{9} \, d$     & $-\frac{64}{15} \, d$     \\ \hline
    \end{tabular}
    \caption{The invariants of $B_{p}$ and $\Xt_{p}$ for the sporadic algebraically primitive examples $E_7$ and $E_8$.
    We write $d = \v\SL(2,\F_{p^g})\v = p^g(p^{2g}-1)$ for the degree of the congruence cover $B_{p} \rightarrow C$.}
    \label{tab:sporadics}
    \end{table}

\begin{proof}
We start by determining the topology of the congruence cover $B_p$.
Let $g$ denote the genus of $X$.
Note that the $E_7$ and $E_8$ curves have generators that are also staircase-shaped.
Either following the proof of \autoref{lem:staircase-heights} or by explicitly calculating the heights $\mathbf{h}$ via the Thurston--Veech construction, we find that the horizontal heights are odd integer polynomials in $\mu$ and the vertical heights are even integer polynomials in $\mu$.
A variant of \autoref{cor:staircases-satisfy-criterion} then implies that, for primes $p$ for which $m_{\mu^2}(x)$ is irreducible, we have
\[
    \Im(\SL(X, \omega) \to \Sp(H_1(X; \F_p))) \cong \SL(2, \F_{p^g}).
\]
Since $-I \notin \SL(X, \omega)$ and $(X, \omega)$ has no translation automorphisms, the deck group of the congruence cover is isomorphic to $\SL(2, \F_{p^g})$ (see \autoref{subsec:topology-of-base}).
We write $d \coloneqq | \SL(2, \F_{p^g}) |$ for the degree of the congruence cover.

The hyperbolic orbifold $C = \widehat{C}$ has two cusps and one orbifold point of order 9 or 15; write $n$ for this orbifold order.
Each cusp maps to an element of order $p$ and thus has $d/p$ preimages on $B_p$.
Riemann--Hurwitz then lets us compute the genus $b$ of $B_p$ by
\begin{align*}
    2 - 2b - 2d/p &= d(2 - 2(0) - 2 - 1 + 1/n)\\
    \implies b &= 1 + \frac d 2(1 - 1/n - 2/p).
\end{align*}

We now compute the total twisting $T$.
Per the Thurston--Veech construction, the horizontal direction decomposes into $4$ cylinders and the vertical direction decomposes into $g$ cylinders.
By construction, the horizontal and vertical multi-twists do one twist in each cylinder.
We conclude that $T = (g + 4) d$.

We now use \autoref{prop:euler_char} and \autoref{prop:signature_formula} to compute the Euler characteristic and signature.
For the latter, we use $\kappa_{(1, 3)} = 7/16$ and $\kappa_{(6)} = 4/7$.

By \autoref{prop:pi_1_suff_cond} we have that $\pi_1(\Xt) \isom \pi_1(\overline{B}_p)$.
Our formula for the genus of $B_p$ in both cases shows that it is never zero, so that $\Xt_p$ is a minimal general type surface by \autoref{thm:kodaira-dimension-genus-high}.
As for the BMY inequality, we appeal to \autoref{prop:alg_prim_BMY}.
\end{proof}

\newpage
\printbibliography
\end{document}